\documentclass[12pt]{article}
\usepackage{graphicx, amssymb, amsthm, amsfonts, latexsym, epsfig, amsmath, amscd, amsbsy}
\usepackage[usenames]{color}
\usepackage[dvipsnames]{xcolor}
\usepackage{tikz}
\usepackage{enumerate}
\usepackage{epstopdf}

\usepackage{adjustbox}

\usetikzlibrary{decorations.pathmorphing}

\theoremstyle{plain}
\newtheorem{theorem}{Theorem}[section]
\newtheorem{lem}[theorem]{Lemma}

\newtheorem{prop}[theorem]{Proposition}
\theoremstyle{definition}

\newtheorem{df}[theorem]{Definition}
\newtheorem{rem}[theorem]{Remark}

\newtheorem*{claim*}{Claim}

\def\R{{\mathbb R}}
\def\N{{\mathbb N}}
\def\Z{{\mathbb Z}}
\def\d{{\mathbb D}}
\def\T{{\mathbb T}}

\def\A{{\mathcal A}}
\def\M{{\mathcal M}}

\def\B{{\mathcal B}}
\def\V{{\mathcal V}}

\def\U{{\mathcal U}}

\def\K{{\mathcal K}}
\def\D{{\mathcal D}}

\def\WY{{\mathcal WY}}
\def\eps{\varepsilon}
\renewcommand{\phi}{\varphi}

\def\Int{\mathop\mathrm{Int}}
\def\Cl{\mathop\mathrm{Cl}}

\def\length{\mathop\mathrm{length}}

\newcommand{\diam}{\operatorname{diam}}

\newcommand{\invlim}{\displaystyle\lim_{\longleftarrow}}

\renewcommand{\setminus}{\smallsetminus}
\def\nin{\notin}
\renewcommand{\hat}{\widehat}
\renewcommand{\tilde}{\widetilde}

\def\bar{\overline}

\newcommand{\vecv}{\mathbf{v}}
\newcommand{\vecu}{\mathbf{u}}
\newcommand{\vecw}{\mathbf{w}}
\newcommand{\veczero}{\mathbf{0}}

\newenvironment{proofof}{\medskip\noindent{\emph{Proof}}}{ \hfill\qed\\ }

\begin{document}
	
	\title{Densely branching trees as models for H\'enon-like and Lozi-like attractors
		\footnote{Part of the research was carried out during authors' stay at the Mathematisches Forschungsinstitut Oberwolfach in September of 2020. 
			The hospitality and tireless support of the MFO staff during the world COVID-19 pandemic is especially gratefully acknowledged.}}
	\author{J.\ Boro\'nski
		\thanks{Supported by the National Science Centre, Poland (NCN), grant no. 2019/34/E/ST1/00237: ``Topological and 
			Dynamical Properties in Parameterized Families of Non-Hyperbolic Attractors: the inverse limit approach''.}
		and
		S.\ \v{S}timac 
		\thanks{Supported in part by the Croatian Science Foundation grant IP-2018-01-7491 and by the program Excellence Initiative – Research University at the Jagiellonian University in Kraków.}}
	\date{}
	
	\maketitle
	
	\begin{abstract}
		Inspired by a recent work of Crovisier and Pujals on mildly dissipative diffeomorphisms of the plane, we show that H\'enon-like and Lozi-like maps on their strange attractors are conjugate to natural extensions (a.k.a. shift homeomorphisms on inverse limits) of maps on metric trees with dense set of branch points. In consequence, these trees very well approximate the topology of the attractors, and the maps on them give good models of the dynamics. To the best of our knowledge, these are the first examples of canonical two-parameter families of attractors in the plane for which one is guaranteed such a 1-dimensional locally connected model tying together topology and dynamics of these attractors. For the H\'enon maps this applies to a positive Lebesgue measure parameter set generalizing the Benedicks-Carleson parameters, the Wang-Young parameter set, and sheds more light onto the result of Barge from 1987, who showed that there exist parameter values for which H\'enon maps on their attractors are not natural extensions of any maps on branched 1-manifolds. For the Lozi maps the result applies to an open set of parameters given by Misiurewicz in 1980. Our result can be seen as a generalization to the non-uniformly hyperbolic world of a classical result of Williams from 1967. We also show that no simpler 1-dimensional models exist.
	\end{abstract}	
	
	
	
	\baselineskip=18pt
	
	\section{Introduction}\label{sec:intro}
	
	Parametric families of maps have been the subject of intense research for many decades, but their dynamical properties are still far from being well understood. The most fundamental example in dimension 2 is the H\'enon family 
	$$H_{a,b} : (x,y) \mapsto (1+y-ax^2, bx),$$
	with the parameter $(a,b)\in \mathbb{R}^2$, where $b\neq 0$. It was introduced by H\'enon \cite{H} in 1976. In 1991, Benedicks and Carleson 
	\cite{BK} showed that there exists a set $\mathcal{BC}$ of positive Lebesgue measure of values of the parameters $(a,b)$ such that $H_{a,b}$ 
	with $(a,b) \in \mathcal{BC}$ exhibits a strange attractor.	
	These attractors are non-uniformly hyperbolic and contain a dense orbit. In 1993, their approach was generalized to a wider class of so-called 
	H\'enon-like maps by Mora and Viana \cite{MV}. In 2001, Wang and Young in \cite{WY} gave simple conditions for dissipative maps that guarantee the existence of strange attractors. They developed a dynamical picture for the attractors in this class, including the geometry of fractal critical sets, nonuniform hyperbolic behavior, and many other properties associated with chaos. Their results hold for the H\'enon family of maps for a positive measure set of parameters $(a, b)$, arbitrarily near $(a^*, 0)$, where $a^* \in [1.5, 2]$ is such that $q_{a^*}(x) = 1 - a^*x^2$ is a Misiurewicz map. Note that these results are valid for both $b > 0$ and $b < 0$. We call this set of parameters the \emph{Wang-Young parameter set} and denote by $\WY$ the Wang-Young parameters with $b > 0$. Note that the Benedicks-Carleson results in \cite{BK} is a version of the case for $a^* = 2$ and $b > 0$.
	
	In 1978, Lozi \cite{L} introduced the family of piecewise affine maps $$L_{a,b} : (x,y) \mapsto (1 + y - a|x|, bx),$$ as a certain simplification 
	of the H\'enon family. In 1980, Misiurewicz \cite{M} proved the existence of strange attractors of the Lozi maps for a large set of parameters, and in \cite{MS2} this set is slightly extended
	to be
	$\M = \{ (a,b) \in \R^2 : b > 0, \ a\sqrt{2} - b > 2, \ 2a + b < 4 \}$. In 2018, Misiurewicz and the second author generalized that approach to a wider class of so-called Lozi-like maps \cite{MS2}, satisfying two  assumptions: $(L3)$ about transversality of the initial segments of the stable and unstable manifolds of one of the fixed points, and $(L4)$ that guarantees enough stretching in the unstable direction (see Section \ref{sec:sm} of the present paper for details). 
	
	Parametric families of maps in the plane, and especially the H\'enon and Lozi maps, have been widely studied in many different contexts and a lot of important
	results have been proved; see for example  \cite{BV}, \cite{Berger}, \cite{I},  \cite{I2}, \cite{IS}, \cite{MS},
	\cite{O}, \cite{WY}, \cite{Y}.	
	
	The H\'enon family and the Lozi family can be also seen as a natural generalization to dimension 2 of the quadratic and the tent family in dimension 1, respectively. It is well known, however, that there is no simple analogy in the study of the corresponding families of 2-dimensional maps, and a surprising number of complications and obstacles arise, which make the study in dimension 2 much more difficult. Nonetheless, in studying dissipative maps on manifolds, one often tries to reduce a given problem to a lower dimensional one, as a stepping stone to obtain results in higher dimensions. 
	
	In 2017, Crovisier and Pujals \cite{CP} showed that for any $a \in (1,2)$ and  $b \in (-\frac{1}{4},0) \cup (0,\frac{1}{4})$ the H\'enon 
	map $H_{a,b}$ is mildly dissipative\footnote{Note that in \cite{CP} the term \emph{strongly dissipative} was used (instead of the term mildly dissipative), but then in \cite{CPT} the authors changed the terminology from strongly dissupative to mildly dissipative, since it better fits the wider context of dissipative maps; e.g. the assumption $0<|b|<\frac{1}{4}$ for H\'enon maps seems rather mild.} on the surface $\d = \{ (x, y) : |x| < 1/2 + 1/a, |y| < 1/2 - a/4 \}$ and has a 1-dimensional structure in the following sense: using 1-dimensional pieces of stable manifolds of $\mu$-almost every point in the surface of dissipation, for any ergodic measure $\mu$ not supported on a hyperbolic sink, a reduced 1-dimensional dynamics is obtained as a continuous non-invertible map acting on 
	a real tree. In particular, there exists a semi-conjugacy $\pi : (\d, H_{a,b}) \to (T', h)$ to a continuous map $h$ on a compact real tree $T'$, which induces an injective map on the set of non-atomic ergodic measures $\mu$ of  $H_{a,b}$, and such that the entropies of $\mu$ and $\pi_*(\mu)$ are the same. The real tree is understood in \cite{CP} as a uniquely path connected metric space. One can see that the trees $T'$ in \cite{CP} are in fact also locally connected (cf. proof of Lemma \ref{t:sc} below), which makes them metric trees\footnote{Note that Shape Theory provides a handful of examples of uniquely path connected metric spaces that are not locally connected, and even separate the plane, whereas compact metric trees (a.k.a. dendrites) are plane non-separating.}. 
	
	In this paper, by focusing on a single stable manifold of the fixed point contained in the strange attractor of the Lozi-like map, we are able to extend Crovisier-Pujals' approach and prove that the Lozi-like map restricted to its strange attractor is in fact conjugate to the natural extension of a map on a metric tree. 
	\begin{theorem}\label{thm:conjugacy}
		Let $\Lambda$ be the strange attractor of the Lozi-like map $F : \mathbb{R}^2 \to \mathbb{R}^2$ that satisfies $(L3)$ and $(L4)$. Then there exists a metric tree $\T$ and a continuous map $f : \T \to \T$ such that $F|_{\Lambda}$ is conjugate to the shift homeomorphism $\sigma_f$ on $\invlim(\T, f)$.
	\end{theorem}
Note that the Lozi family of maps within the Misiurewicz set of parameters $\M$ is a subfamily of the Lozi-like maps that satisfy $(L3)$ and $(L4)$. Stating the results and carrying out the proofs for Lozi-like maps has the advantage of being more general and focusing on the essential properties which make the proofs work, instead of specific details of the Lozi maps.

	Note also that there is a bijection between the set of invariant probability measures of a map and the shift homeomorphism it induces on its inverse limit space \cite{Rohlin}. This means that the metric tree $\T$ with its map $f$ gives a very good model for dynamics of the Lozi-like map $F|_\Lambda$. This is one of the reasons why one would like to understand better the geometry of the tree $\T$. In that context we prove the following.
	\begin{theorem}\label{thm:dense}
		The set of branch points of $\T$ is a dense subset of $\T$.
	\end{theorem} 
	From topological point of view, the above two theorems imply that the Lozi-like attractor $\Lambda$ is in a well-defined sense very well approximated by the densely branching metric tree $\T$; i.e. 
	{\it for every $\epsilon>0$ there exists a continuous surjection $g_\epsilon:\Lambda\to\T$ such that $\diam(g_\epsilon^{-1}(t))<\epsilon$ for all $t\in\T$.} 
	Indeed, since $\Lambda=\invlim(\T, f)$ in order to obtain the map $g_\eps$ it is enough to take the projection from $\invlim(\T, f)$ onto the $i$th coordinate space, for $i>0$ large enough. 
	The density of branch points means that topologically $\T$ has a similar structure to several known fractals, such as the antenna set \cite{BiTy}, Hata’s tree-like set \cite{Kig}, and the continuum self-similar tree \cite{BM}. It also resembles Wa\.zewski universal dendrite, which is a universal object in the category of all metric trees \cite{Wazewski}. 
	
	Our approach is not restricted to the $C^0$ world but extends to the H\'enon family when parameters belong to the set of parameters $\mathcal{WY}$. Thus, for the H\'enon maps we prove the following.

	\begin{theorem}\label{cor:conjugacy}
		Let $(a,b) \in \WY$ and $\Lambda'$ be the strange attractor of the H\'enon map $H_{a,b} : \mathbb{R}^2 \to \mathbb{R}^2$. Then there exists a metric tree $\T'$ and a continuous map 
		$h : \T' \to \T'$ such that $H_{a,b}|_{\Lambda'}$ is conjugate to the natural extension $\sigma_h$ on $\invlim(\T', h)$.
	\end{theorem}
	\begin{theorem}\label{cor:dense}
		The set of branch points of $\T'$ is a dense subset of $\T'$.
	\end{theorem} 
	Again, this leads to the conclusion that densely branching tree $\T'$ well approximates the H\'enon attractor $\Lambda'$; i.e. {\it for every $\epsilon>0$ there exists a continuous surjection $g'_\epsilon:\Lambda\to\T'$ such that $\diam({g'}_\epsilon^{-1}(t))<\epsilon$ for all $t\in\T'$.} 
	
	We also show that our results are optimal in the following sense.
	\begin{theorem}\label{thm:LO}
		Let $\Lambda$ be the strange attractor of the Lozi-like map $F : \mathbb{R}^2 \to \mathbb{R}^2$ that satisfies $(L3)$ and $(L4)$. If $f : \T \to \T$ is a tree map such that $(\invlim (\T, f), \hat f)$ is conjugate to $(\Lambda, F)$ then $\T$ is densely branching.
	\end{theorem}
The proof of Theorem \ref{thm:LO} works verbatim for the H\'enon maps within parameter set $\WY$. 
\begin{theorem}\label{thm:HO}
	Let $(a, b) \in \WY$ and $\Lambda$ be the strange attractor of the H\'enon map $H_{a,b} : \mathbb{R}^2 \to \mathbb{R}^2$. If $f : \T \to \T$ is a tree map such that $(\invlim (\T, f), \hat f)$ is conjugate to $(\Lambda_{H_{a,b}}, H_{a,b})$ then $\T$ is densely branching.
\end{theorem}
	
	It seems surprising to us that models for the H\'enon and the Lozi maps can be obtained in such a similar way, since although the relationship between the quadratic maps and the tent maps is very strong and well known, not much about the relationship between the H\'enon and the Lozi maps has been known so far. 
	
	To the best of our knowledge, these are the first examples of canonical two-parameter families of attractors in the plane for which one is guaranteed such a 1-dimensional locally connected model tying together topology and dynamics of these attractors. Moreover, even though for the purpose of exposition we have confined our attention to the H\'enon and the Lozi-likr maps, our approach is by no means limited to these two families of maps. It will be clear that all statements generalize to the H\'enon-like maps as well. 
	
	Our results can be seen as complementing the recent advances done by Boyland, de Carvalho and Hall on natural extensions and inverse limits of parametric families of 1-dimensional maps \cite{BCH}, \cite{BCH2}, \cite{BCH3}, \cite{BCH4}, \cite{BCH5}. One can also see them as a generalization to the non-uniformly hyperbolic world of a classical result of Williams \cite{W} from 
	1967, who showed that manifold diffeomorphisms on hyperbolic attractors can be conjugated to the shift on an inverse 
	limit of branched 1-manifolds. In 1995 Barge \cite{B2} proved that for parameter values $a$ where the quadratic map $q_a$ has an attracting 
	periodic orbit and small values of $b$, the H\'enon maps $H_{a,b}$, when restricted to their attracting sets, are topologically conjugate to shift
	homeomorphisms on inverse limits on an interval with bonding map $q_a$. However, earlier in 1987 Barge \cite{B} showed that there are parameter values 
	for which Williams' result no longer holds. For example, for values of $a$ for which the critical point of the quadratic map eventually lands on 
	a repelling periodic orbit there are arbitrarily small values of $b$ for which $H_{a,b}$ has a hyperbolic periodic orbit with homoclinic tangencies.
	For such an $a$ and $b$, the dynamics on H\'enon attractors is not conjugate to a natural extension of any map on a branched 1-manifold. 
	
	The paper is organized as follows. In Section \ref{sec:sm} we recall basic facts about the Lozi-like maps and define V-points and their levels, which will later play an important role in proving dense branching of trees for the Lozi-like attractors. In Section \ref{sec:1d} we define the family $\Gamma$ of certain arcs contained in the stable manifold of the fixed point $X$ in a Lozi-like attractor, and show that this family has certain properties required in \cite{CP} to guarantee a reduction to a tree. We also show how to choose an analogous family $\Gamma'$ for the H\'enon family with parameters in $\WY$. Then we prove Theorem \ref{thm:conjugacy} and Theorem \ref{cor:conjugacy}. In Section \ref{geo} we consider the geometry of the trees from Theorem \ref{thm:conjugacy} and Theorem \ref{cor:conjugacy}, define their stems and level of stems, and show that the trees are densely branching, therefore proving Theorem \ref{thm:dense} and Theorem \ref{cor:dense}. We also show that our results are optimal by proving Theorem \ref{thm:LO}. Finally, in Section \ref{top} we observe that the H\'enon and Lozi attractors considered here are topologically indecomposable.

\section{Preliminaries on the Lozi-like maps}\label{sec:sm} 
	
In \cite{MS2}, Misiurewicz and the second author of the present paper defined and studied a family of maps generalizing the Lozi maps, which they called the Lozi-like maps. In this section, we recall the definition of the Lozi-like maps and some of their properties needed in this paper.
	
A \emph{cone} $K$ in $\R^2$ is a set given by a unit vector $\vecv$ and a number $\ell\in(0,1)$ by
\begin{equation}\label{econe1}
	K=\{\vecu\in\R^2:|\vecu\cdot\vecv|\ge \ell\|\vecu\|\},
\end{equation}
where $||\cdot||$ denotes the usual Euclidean norm. The straight line $\{t\vecv:t\in\R\}$ is the \emph{axis} of the cone. Two cones are called \emph{disjoint} if their intersection consists only of the vector $\veczero$.	A pair of cones $K^u$ and $K^s$ is called \emph{universal} if they are disjoint and the axis of $K^u$ is the $x$-axis, and the axis of $K^s$ is the $y$-axis. For an open set $U \subseteq \R^2$, a \emph{cone-field} $C$ on $U$ is the assignment of a cone $K_P$ to each point $P \in U$ such that the axis $\vecv(P)$ and the coefficient $\ell(P)$ vary continuously with $P$.

\begin{df}(\cite[Definition 2.5]{MS2})\label{df:sh}
	Let $F_1, F_2 : \R^2 \to \R^2$ be $C^1$ diffeomorphisms. We say that
	$F_1$ and $F_2$ are \emph{synchronously hyperbolic} if they are either
	both order reversing, or both order preserving, and there exist
	$\lambda > 1$, a universal pair of cones $K^u$ and $K^s$, and cone
	fields $C^u$ and $C^s$ (consisting of cones $K^u_P$ and $K^s_P$, $P
	\in \R^2$, respectively) which satisfy the following properties:
	\begin{enumerate}
		\item[(S1)] For every point $P \in \R^2$ we have $K^u_P \subset K^u$,
		$K^s_P \subset K^s$, ${DF_i}_P(K^u_P) \subset K^u_{F_i(P)}$, and
		$D{{F_i}_P^{-1}}(K^s_P) \subset K^s_{F_i^{-1}(P)}$, for $i = 1,2$.
		\item[(S2)] For every point $P \in \R^2$ and $i = 1,2$ we have
		$||DF_i(\vecu)|| \ge \lambda||\vecu||$ for every $\vecu \in K^u_P$
		and $||DF^{-1}_i(\vecw)|| \ge \lambda||\vecw||$ for every $\vecw \in
		K^s_P$.
		\item[(S3)] There exists a smooth curve $\kappa \subset \R^2$ such
		that for every $P \in \kappa$ we have $F_1(P) = F_2(P)$, the vector
		tangent to $\kappa$ at $P$ belongs to $K^s_P$, and the vector
		tangent to $F_i(\kappa)$ at $F_i(P)$ belongs to $K^u_{F_i(P)}$. We
		require that $\kappa$ is infinite in both directions.
	\end{enumerate}
	We call $\kappa$ the \emph{divider}. It divides the plane into two
	parts which we call the \emph{left half-plane} and the \emph{right
		half-plane}. Also $F_1(\kappa) = F_2(\kappa)$ divides the plane into
	two parts which we call the \emph{upper half-plane} and the
	\emph{lower half-plane}.
\end{df}

\begin{rem}\label{newrem}
	Since $F_1$ and $F_2$ are either both order reversing, or both order
	preserving, for any $P \in \R^2$, $F_1(P)$ and $F_2(P)$ belong to the
	same (upper or lower) half-plane. Without loss of generality we 
	assume that $F_i$, $i = 1,2$, maps the left half-plane onto the lower
	one and the right half-plane onto the upper one.
\end{rem}

\begin{rem}\label{rem:new}
Since the existence of the invariant cone fields implies
hyperbolicity (see~\cite[Proposition~5.4.3] {BS}), if $F_1$ and
$F_2$ are synchronously hyperbolic then both are hyperbolic (with
stable and unstable directions of dimension 1). Also, for each of
them, by (S1) and \cite[Lemma 2.2]{MS2} the stable and unstable manifolds
of any point are infinite in both directions.
\end{rem}

Recall that for a map $f$ a \emph{trapping region} is a nonempty set
that is mapped with its closure into its interior. A set $A$ is an
\emph{attractor} if it has a neighborhood $U$ which is a trapping
region, $A =\bigcap_{n=0}^\infty f^n(U)$, and $f$ restricted to $A$ is
topologically transitive.

\begin{df}(\cite[Definition 2.7]{MS2})\label{df:ll}
	Let $F_1, F_2 : \R^2 \to \R^2$ be synchronously hyperbolic $C^1$
	diffeomorphisms with the divider $\kappa$. Let $F : \R^2 \to \R^2$ be
	defined by the formula
	\[
	F(P) =
	\begin{cases}
		F_1(P), & \textrm{if $P$ is in the left  half-plane,}\\
		F_2(P), & \textrm{if $P$ is in the right half-plane.}
	\end{cases}
	\]
	We call the map $F$ \emph{Lozi-like} if the following hold:
	\begin{enumerate}
		\item[(L1)] $-1 < \det DF_i(P) < 0$ for every point $P \in \R^2$ and
		$i = 1, 2$.
		\item[(L2)] There exists a trapping region $\D$ (for the map $F$),
		which is homeomorphic to an open disk and its closure is
		homeomorphic to a closed disk.
	\end{enumerate}
\end{df}

Observe that by Remark~\ref{newrem}, $F$ is a homeomorphism of $\R^2$ onto itself. Obviously, the Lozi maps $L_{a, b}(x, y) = (1 + y - a|x|, bx)$, with $(a,b) \in \M = \{ (a,b) : b > 0, \ a\sqrt{2} > b+2, \ 2a + b < 4 \}$, provide an example of Lozi-like maps with $y$-axis as the divider. For the set $\D$ we take a neighborhood of the triangle as in~\cite{M} that usually serves as the trapping region for the Lozi map.

Let us recall some properties of a Lozi-like map $F$.

\begin{lem}\emph{(\cite[Lemma 2.8]{MS2})}\label{lem:fixed point}
	\begin{enumerate}
		\item[$(P1)$] $\D \cap \kappa \ne \emptyset$ and consequently
		$\D \cap F(\kappa) \ne \emptyset$.
		\item[$(P2)$] There exists a unique fixed point in $\D$.
	\end{enumerate}
	Let us denote this fixed point by $X$. We assume without loss of
	generality that it belongs to the right half-plane.
	\begin{enumerate}
		\item[$(P3)$] Let $\lambda_1$ and $\lambda_2$ denote the eigenvalues
		of $DF(X)$ with $|\lambda_1| > 1$ and $|\lambda_2| < 1$. Then
		$\lambda_1 < -1$ and $\lambda_2 > 0$.
		\item[$(P4)$] Let $\hat \lambda_1$ and $\hat \lambda_2$ denote the
		eigenvalues of $DF(P) = DF_1(P)$ when $P$ is in the left half-plane,
		and $|\hat \lambda_1| > 1$, $|\hat \lambda_2| < 1$. Then $\hat
		\lambda_1 > 1$ and $\hat \lambda_2 < 0$.
		\item[$(P5)$] $X \ne \kappa \cap F(\kappa)$ and consequently $X \nin
		\kappa \cup F(\kappa)$.
	\end{enumerate}
\end{lem}
Let $\Lambda := \bigcap_{n=0}^\infty F^n(\D)$. Observe that $\Lambda$ is completely invariant, that is, $\Lambda=F(\Lambda)=F^{-1}(\Lambda)$.

Note that the Lozi-like map $F$ is not everywhere differentiable, but $-1 <\det DF(P) < 0$ at every point $P$ where the derivative exists. Also, its hyperbolic structure can be
understood only as the existence of a hyperbolic splitting at those points at which it exists (for which the derivative exists at the whole trajectory). This splitting cannot be extended to a continuous one on the whole plane. The usual meaning of the notions of stable and unstable manifolds also has to be changed a little bit. They are broken curves, and 
therefore not manifolds. Nevertheless, we prefer to call them (un)stable manifolds rather than (un)stable sets. By the definition of $F$, they exist at almost all points of the trapping region and are infinite in both directions.

We will denote the unstable and stable manifold of a map $f$ at a point $P$ by $W^u_f(P)$ and $W^s_f(P)$ respectively. Also, if $A$ is an arc, or an arc-component and $P, Q \in A$, $P \ne Q$, we denote by $[P, Q] \subset A$ a unique arc of $A$ with boundary points $P$ and $Q$, and $(P, Q) = [P, Q] \setminus \{ P, Q \}$ will be also sometimes called an arc. Those sets will be usually subsets of $W^u_F(X)$ or $W^s_F(X)$, and if they are, we will denote them as $[P, Q]^u$ or $[P, Q]^s$, respectively. For a point $P = (P_x, P_y) \in \R^2$ we denote $F^n(P) = P^n = (P^n_x, P^n_y)$, $n \in \Z$, and consequently $P = P^0$. We will call the four regions of the plane given by $\kappa$ and $F(\kappa)$ the \emph{quadrants}, and their order is the usual one. 

Note first that $\kappa \cap W^s_F(X) \ne \emptyset$. Let us denote by $V$ the point where $W^s_{F_2}(X)$ intersects $\kappa$. Note that $V$ is also a point of intersection of $W^s_{F}(X)$ and $\kappa$. 

Recall that $\kappa$ and $W^u_{F_2}(X)$ intersect at exactly one point. Therefore, $F(\kappa)$ and $W^u_{F_2}(X)$ also intersect at exactly one point, denote it by $Z$. Let us consider the arc $[Z^{-1}, Z]^u \subset W^u_{F_2}(X)$, see Figure~\ref{fig.BP1}. Note that $Z^{-1} = F_2^{-1}(Z) \in \kappa$ and $X \in [Z^{-1}, Z]^u \subset W^u_F(X)$. Hence $Z$ belongs to the right half-plane and since the stretching factor $\lambda$ is larger than $1$ (see Definition~\ref{df:sh}), $Z^1$ belongs to the second quadrant. Therefore $Z^2$ lies in the lower half-plane. In order to prove that $\Lambda$ is the attractor for $F$, we restrict the possible position of $Z^2$, and increase the lower bound on the stretching factor as follows:
\begin{enumerate}
	\item[$(L3)$] $[Z^2, Z]^u \subset W^u_F(X)$ intersects $[X, V]^s \subset W^s_F(X)$,
	\item[$(L4)$] The stretching factor $\lambda$ is larger than $\sqrt2$.
\end{enumerate}
The above conditions are natural in the sense that any Lozi map $L_{a,b}$ with $(a, b) \in \M$ satisfies them.
\begin{figure}[h]
	\begin{center}
		\begin{tikzpicture}[scale=0.75]
			\tikzstyle{every node}=[draw, circle, fill=white, minimum size=1.5pt, 
			inner sep=0pt]
			\tikzstyle{dot}=[circle, fill=white, minimum size=0pt, inner sep=0pt, outer sep=-1pt]
			
			\node[dot, draw=none, label=above: \tiny $Z^{-1}$] at (0.2,2.1) {};
			\node[dot, draw=none, label=above: \tiny $V$] at (0.1,-3.4) {};
			\node[dot, draw=none, label=above: \tiny $V^1$] at (1,0.5) {};
			
			\node[label=above: \tiny $Z^3$] at (-3.73,-1.64) {};
			\node[] at (2.33,0.01) {};
			\node[label=above: \tiny $Z^1$] at (-8.04,3.21) {};
			\node[label=above: \tiny $Z$] at (6.5,-0.01) {};
			\node[label=below: \tiny $Z^2$] at (-3.29,-4.02) {};
			\node at (5.72,0) {};
			\node[] at (0.25,1.16) {};
			\node at (4.36,-0.01) {};
			\node[label=below: \tiny $Z^4$] at (-1.98,-2.1) {};
			
			\node[] at (0.37,-3.46) {};
			\node[] at (0.75,-2.35) {};
			\node[label=above: \tiny $X$] at (1.9,1) {};
			\node[] at (-0.28,1.5) {};
			\node[] at (1.55,-0.01) {};
			
			\node[draw=none, label=above: \tiny $\kappa$] at (-0.2,3) {};
			\node[draw=none, label=above: \tiny $F(\kappa)$] at (-8,-0.2) {};
			
			\node[draw=none, label=above: \small $\Delta$] at (-3.73,0.64) {};
			
			\draw[
			decoration={snake,
				amplitude = 0.1mm,
				segment length = 7mm,
				post length=0.9mm},decorate] (-3.73,-1.64) -- (2.33,0);
			\draw[
			decoration={snake,
				amplitude = 0.1mm,
				segment length = 7mm,
				post length=0.9mm},decorate] (2.33,0) -- (-8.04,3.21);
			\draw[
			decoration={snake,
				amplitude = 0.1mm,
				segment length = 7mm,
				post length=0.9mm},decorate] (-8.04,3.21) -- (6.42,0);
			\draw[
			decoration={snake,
				amplitude = 0.05mm,
				segment length = 7mm,
				post length=0.9mm},decorate] (6.42,0) -- (-3.29,-4.02);
			\draw[
			decoration={snake,
				amplitude = 0.05mm,
				segment length = 7mm,
				post length=0.9mm},decorate] (-3.29,-4.02) -- (5.72,0);
			\draw[
			decoration={snake,
				amplitude = 0.03mm,
				segment length = 6mm,
				post length=0.9mm},decorate] (5.72,0) -- (0.25,1.16);
			\draw[
			decoration={snake,
				amplitude = 0.1mm,
				segment length = 7mm,
				post length=0.9mm},decorate] (0.25,1.16) -- (4.36,-0.01);
			\draw[
			decoration={snake,
				amplitude = 0.1mm,
				segment length = 7mm,
				post length=0.9mm},decorate] (4.36,-0.01) -- (-1.98,-2.1);
			\draw[
			decoration={snake,
				amplitude = 0.2mm,
				segment length = 15mm,
				post length=0.9mm},decorate] (-8.5,0) -- (7,0);
			\draw[
			decoration={snake,
				amplitude = 0.2mm,
				segment length = 10mm,
				post length=0.9mm},decorate](-0.5,3.3) -- (0.5,-4.3);
			\draw[
			decoration={snake,
				amplitude = 0.2mm,
				segment length = 10mm,
				post length=0.9mm},decorate](1.9,1) -- (0.37,-3.46);
			
				\draw[gray,dashed](-8.04,3.21)--(-3.29,-4.02);
		\end{tikzpicture}
		\caption{Positions of some distinguished points and the ``triangle'' $\Delta$.}
		\label{fig.BP1}
	\end{center}
\end{figure}

Let us denote by $\Delta$ the ``triangle'' with vertices $Z$, $Z^1$, $Z^2$, and with edges $[Z^1, Z]^u$, $[Z^2, Z]^u$ and a straight line segment that connects $Z^1$ and $Z^2$ (see Figure~\ref{fig.BP1}). By \cite[Lemma 3.2]{MS2}, $F(\Delta) \subset \Delta$ and 
$$\Lambda = \bigcap_{n=0}^{\infty} F^n(\Delta).$$ 

\begin{prop}\emph{(\cite[Proposition 3.4]{MS2})}\label{prop:closure}
	Let a Lozi-like map $F$ satisfy $(L3)$. Then $\Lambda = \Cl
	(W^u_F(X))$.
\end{prop}

\begin{prop}\emph{(\cite[Proposition 3.7]{MS2})}\label{prop:transitivity}
	Let a Lozi-like map $F$ satisfy $(L3)$ and $(L4)$. Then $F|_{\Lambda}$
	is topologically mixing, i.e., for all open subsets $U_1$, $U_2$ of $\R^2$
	such that $U_1 \cap \Lambda \ne \emptyset$ and $U_2 \cap \Lambda \ne
	\emptyset$, there exists $N \in \N$ such that for every $n \ge N$ the
	set $F^n(U_1) \cap U_2 \cap \Lambda$ is nonempty.
\end{prop}

This proves that $\Lambda$ is the strange attractor of $F$. 

From now on we will work only with the Lozi-like maps that satisfy $(L3)$, $(L4)$, and so they have the strange attractors. For simplicity, we will denote the stable and unstable manifolds of the fixed point $X$ of the Lozi-like map $F$ by $W^s_X$ and $W^u_X$, respectively.
	
	Let us consider the structure of the stable manifold $W^s_X$. Let $\phi : \mathbb{R} \to W^s_X$ with $\phi(0) = X$ be its parametrization. One half of $W^s_X$, say $\phi((-\infty,0])$,  starts at $X$ and goes to the infinity in the first quadrant. The other half $\phi([0,\infty))$ intersects the vertical axis for the first time at the point $V$. Note that $F^{-1}$ maps the upper half-plane onto the right one and the lower half-plane onto the left one. 
	
	We call a point $P \in W^s_X$ a {\it V-point} if there exists a $k \in \N_0$ such that $P^k \in \kappa$. We call a point $P$ a {\it  basic V-point} if $P \in \kappa$ and there is no $n \in \N$ such that $P^n \in \kappa$. In this case $W^s_X$ intersects $\kappa$ at $P$ transversely, that is, for every neighborhood $U$ of $P$, the arc of $U \cap W^s_X$ that contains $P$ intersects both, the left and the right half-planes. We say that a V-point $P$ has {\it level} $k+1$ if $P^k$ is a basic V-point. In other words a V-point $P$ has level $k+1$ if it is $k$th preimage of a basic V-point. 
	
	Note that if $P$ is a basic V-point, its image $P^1$ is not a V-point, $P^1$ belongs to $F(\kappa)$ and $W^s_X$ intersects $F(\kappa)$ at $P^1$ transversely. On the other hand, all preimages of V-points are V-points and we say that $W^s_X$ ``makes a turn'', or ``has a V-shape'', at every V-point. If $P^{-n} \in \kappa$ for a basic V-point $P$, then   there exists a small neighborhood of $P^{-n}$ such that the arc of $U \cap W^s_X$ that contains $P^{-n}$ lies either in the right, or in the left half-plane (with the boundary $\kappa$). In this case we say that the V-point $P^{-n}$ intersects $\kappa$ tangentially. Let us denote the set of all V-points by $\V$.
	
	We can think of $W^s_X$ 
	in two ways. In the first one, $W^s_X$
	is a subset of the plane. In the second one, we 
	``straighten'' it out and consider it as
	the real line. Note that the topology in
	$W^s_X$ is different in both cases. When 
	we use the second way, there is a natural order $\prec$ on $W^s_X$
	given by
	$$Q\prec Q'\textrm{ if and only if }\phi^{-1}(Q)<\phi^{-1}(Q').$$
	In particular, $X \prec V$.
	\begin{lem}\label{l:discrete}
		The set $\phi^{-1}(\V)$ is discrete.
	\end{lem}
	\begin{proof}
		Note that $V = V^0$ and $V^{-1}$ are two consecutive V-points. Note also that $F^{-1}|_{W^s_X}$ is expanding and $W^s_X = R \cup (\bigcup_{i=0}^{\infty}F^{-i}([X, V^0]))$, where $R = \phi((-\infty,0])$. If $P$ and $Q$ are two consecutive points of $\V$, then between $P^{-1}$ and $Q^{-1}$ there is at most one point of $\V$. Therefore, by induction we see that in $F^{-i}([V^0,V^{-1}])$ there are only finitely many points of $\V$. This proves that $\phi^{-1}(\V)$ is discrete.
	\end{proof}
	
	\begin{lem}\label{lem:dense}
		$W^s_X$ is dense in $\D$.
	\end{lem}
	\begin{proof}
		The proof is analogous to the proof of \cite[Theorem 6]{M}, but we provide it here for completeness.
		
		By Remark \ref{rem:new}, at almost all points $P \in \D$ there exist local stable and unstable manifolds and broken global stable and unstable manifolds, infinite in both directions. By \cite[Lemma 3.5]{MS2}, for every arc $I \subset \D$ contained in some unstable manifold, there exist $n \in \N_0$ and an arc $I' \subset F^n(I)$ such that $I'$ intersects both $\kappa$ and $F(\kappa)$. By \cite[Remark 3.6]{MS2}, each arc contained in $F(\Delta)$ and intersecting both $\kappa$ and $F(\kappa)$ also intersect $[X, V]^s$. Therefore, the set of homoclinic points of $X$ is dense in $\Lambda$ and there exists an open set $U$, $\Lambda \subset U \subseteq \D$ such that $U \cap W^s_X$ is dense in $U$. Let us suppose now by contradiction that there exists an open set $U' \subset \D$ such that $U' \cap W^s_X = \emptyset$. Since $F(\D) \subset \D$, there exists an $n \in \N$ such that $F^n(U') \subset U$. Therefore, $F^n(U') \cap W^s_X \ne \emptyset$. Since $W^s_X$ is invariant for $F$, $\emptyset \ne F^{-n}(F^n(U') \cap W^s_X) \subset W^s_X$ contradicting the assumption that $U' \cap W^s_X = \emptyset$.
	\end{proof}
	
	\begin{rem}\label{rem:dense}
		Note that the set of all basic V-points belongs to, and is dense in the arc $\kappa \cap \D$. Inductively, the set of V-points of level $k+1$ belongs to, and is dense in the broken arc $F^{-k}(\kappa) \cap \D$. Also, all points of intersections of $W^s_X$ and $F^{-k}(\kappa) \cap \D$ are V-points of level greater than or equal $k+1$. If a level is greater than $k+1$, then $W^s_X$ intersects $F^{-k}(\kappa) \cap \D$ in that point tangentially. 
	\end{rem}
	
	\section{Reduction to a locally connected model}\label{sec:1d}
	
	In this section we will prove Theorems \ref{thm:conjugacy} and \ref{cor:conjugacy}. First, we focus on the Lozi-like maps. Recall, for any point $P$, we let $F^n(P) = P^n$ and consequently $P = P^0$. Also, given a set $D$, we denote the boundary of $D$ by $\partial D$, and the closure of $D$ by $\Cl D$.
	\subsection{Building conjugacy on the Lozi attractors}\label{sec:lm}
	We are interested in the family of arcs $$\Gamma =\{\gamma:\gamma \textrm{ is a connected component of }\Delta \cap W^s_X\}.$$ 
	Let $W = [X, V]^s \cap [Z^2, Z]^u$. Obviously $[X,W]^s \in \Gamma$. Since the set of homoclinic points
	$W^s_X \cap W^u_X$ is countable and dense in $W^u_X$, $\Gamma$ is countably infinite and for every $\gamma \in \Gamma$ 
	the both endpoints of $\gamma$ are contained in $\partial \Delta$. Therefore,
	every $\gamma$ divides $\Delta$ into 
	finitely many (connected) components. Let
	us denote the number of components for
	$\gamma$ by $c(\gamma)$. For simplicity,
	we shall use the term \emph{components for}
	$\gamma$, and that will always mean
	the components of $\Delta\setminus\{\gamma\}$. If $c(\gamma) = 2$, then 
	$\Int \gamma$ (in the topology of the real line, i.e.\ $\Int\gamma=\varphi(\Int\varphi^{-1}(\gamma))$)
	does not intersect $\partial \Delta$. If
	$c(\gamma) > 2$, then 
	$\Int\gamma \cap \partial \Delta$
	consists of $c(\gamma) - 2$ V-points. 
	
	In the next lemma we will prove that every
	$\gamma \in \Gamma$ is the limit of arcs
	in $\Gamma$ and is accumulated on both
	sides.
	
	\begin{lem}\label{lem:gamma}
		Let $\gamma \in \Gamma$. 
		\begin{enumerate}[(1)]
			\item If $c(\gamma) = 2$, then in every
			component for $\gamma$ there exists a
			sequence $(\gamma_i)_{i \in \N}\subset\Gamma$ which 
			converges to $\gamma$ in the Hausdorff
			metric.
			\item Let $c(\gamma) = n > 2$, 
			$\gamma \cap \partial \Delta = 
			\{ P_1, \dots , P_n \}$ and
			$P_1 \prec \dots \prec P_n$. 
			Every arc $[P_k, P_{k+2}] \subset \gamma$,
			$k = 1, \dots , n-2$, divides $\Delta$
			into three components and only one of them
			contains the whole arc $[P_k, P_{k+2}]$ in
			its boundary. In that component there
			exists a sequence  $(\gamma_i)_{i \in \N}\subset\Gamma$
			which converges to the arc $[P_k, P_{k+2}]$
			in the Hausdorff metric. The arc 
			$[P_1, P_2] \subset \gamma$ divides 
			$\Delta$ into two components and one 
			of them doesn't intersect $\gamma$. In that
			component there exists a sequence 
			$(\gamma'_i)_{i \in \N}\subset\Gamma$ which converges to
			the arc $[P_1, P_2]$ in the Hausdorff
			metric. An analogous statement
			holds for the arc 
			$[P_{n-1}, P_n] \subset \gamma$.
		\end{enumerate}
	\end{lem}
	\begin{proof}
		Let $T_0$ denote the triangle with
		vertices $X$, $V^0$ and 
		$A^0 := L^{-1}(W)$. Note that $A^0 = [V^0, V^{-1}] \cap [X, Z^1]$, where
		$[V^0, V^{-1}] \subset W^s_X$, 
		$[X, Z^1] \subset W^u_X$. The edges of $T_0$ are
		$[X, V^0], [V^0, A^0] \subset W^s_X$
		and $[X, A^0] \subset W^u_X$ (see Figure \ref{fig:lozi2}, left). Denote by
		$\partial^s T_0 := [X, V^0] \cup [V^0, A^0]$ and $\partial^u T_0 := [X, A^0]$. 
		
		Let us consider the polygons 
		$T_n := L^{-n}(T_0)$, $n \in \N_0$, as well as 
		their stable and unstable boundaries
		$\partial^s T_n := L^{-n}(\partial^s T_0)$ and $\partial^u T_n := L^{-n}(\partial^u T_0)$ respectively. Then 
		$\partial^s T_n$ is a polygonal line in $W^s_X$, $\partial^s T_{n-1} \subset \partial^s T_n$ since $X \in \partial^s T_n$ for every $n \in \N$, and 
		$\partial^u T_n$ is a segment in $W^u_X$. 
		Also, $\length (\partial^u T_n) \to 0$
		when $n \to \infty$, since 
		$A^0 \in W^u_X$ and
		$A^{-n} \to X$ when $n \to \infty$. 
		
			\begin{figure}[ht]
			
			\begin{adjustbox}{center}
				\begin{tabular}{cc}
					\hspace{-0.5cm}
					\resizebox{0.4\textwidth}{!}{
						
						\begin{tikzpicture}[scale=1]
							\tikzstyle{every node}=[draw, circle, fill=white, minimum size=1.5pt, inner sep=0pt]
							\tikzstyle{dot}=[circle, fill=white, minimum size=0pt, inner sep=0pt, outer sep=-1pt]
							\node (n1) at (2*4/3,0) {};
							\node (n2) at (-2*4/3,2*2/3)  {};
							\node (n3) at (-2*2/3,-2*2/3)  {};
							\node (n4) at (2*4/9,2*2/9) {};
							\node (n5) at (0.62,0.51)  {};
							\node (n10) at (0,-2*4/3)  {};
							\node (n11) at (0.55,-0.7)  {};
							\node (n12) at (-4/3,1)  {};
							
							\node (n01) at (0,2*1/3)  {};
							\node (n02) at (1.6,0.27)  {};
							\node (n03) at (-1.7,2)  {};
							
							\tikzstyle{every node}=[draw, circle, fill=white, minimum size=0.1pt, inner sep=0pt]
							\node (n6) at (-3,0) {};
							\node (n7) at (3,0)  {};
							\node (n8) at (0,2)  {};
							\node (n9) at (0,-5.3)  {};
							
							\foreach \from/\to in {n1/n2, n1/n3, n6/n7, n8/n9}
							\draw[decoration={snake,
								amplitude = 0.1mm,
								segment length = 7mm,
								post length=0.9mm},decorate] (\from) -- (\to);
							
							\draw[gray,dashed](n2)--(n3);
							
							\node[dot, draw=none, label=above: \tiny $Z$] at (2*4/3,0) {};
							\node[dot, draw=none, label=below: \tiny $Y^0$] at (0,2*1/3) {};
							\node[dot, draw=none, label=below: \tiny $A^0$] at (-1,0.9) {};
							\node[dot, draw=none, label=below: \tiny $Z^1$] at (-2*4/3,2*2/3) {};
							\node[dot, draw=none, label=above: \tiny $Z^2$] at (-2*2/3,-2*2/3) {};
							\node[dot, draw=none, label=below: \tiny $X$] at (2*4/9,2*2/9) {};
							\node[dot, draw=none, label=above: \tiny $V^0$] at (0.2,-2*4/3) {};
							\node[dot, draw=none, label=below: \tiny $A^{-1}$] at (1.5,0.26) {};
							\node[dot, draw=none, label=below: \tiny $A^{-2}$] at (n5) {};
							\node[dot, draw=none, label=below: \tiny $V^{-1}$] at (-1.3,1.9) {};
							\node[dot, draw=none, label=above: \tiny $W$] at (0.7,-0.79) {};
							
							\draw[red,thick,decoration={snake,
								amplitude = 0.1mm,
								segment length = 7mm,
								post length=0.9mm},decorate](n4)--(n10);
							\draw[red,thick,decoration={snake,
								amplitude = 0.1mm,
								segment length = 7mm,
								post length=0.9mm},decorate](n10)--(n12);
							
							\draw[Purple,thick,decoration={snake,
								amplitude = 0.1mm,
								segment length = 7mm,
								post length=0.9mm},decorate](n12)--(-1.7,2);
							\draw[Purple,thick,decoration={snake,
								amplitude = 0.1mm,
								segment length = 7mm,
								post length=0.9mm},decorate](-1.7,2)--(0,-4);
							\draw[Purple,thick,decoration={snake,
								amplitude = 0.1mm,
								segment length = 7mm,
								post length=0.9mm},decorate](0,-4)--(1.6,0.26);
							
							\draw[orange,thick,decoration={snake,
								amplitude = 0.1mm,
								segment length = 7mm,
								post length=0.9mm},decorate](1.6,0.26)--(2.45,2.5);
							\draw[orange,thick,decoration={snake,
								amplitude = 0.1mm,
								segment length = 7mm,
								post length=0.9mm},decorate](2.45,2.5)--(0,-4.5);
							\draw[orange,thick,decoration={snake,
								amplitude = 0.1mm,
								segment length = 7mm,
								post length=0.9mm},decorate](0,-4.5)--(-1.96, 2.8);
							\draw[orange,thick,decoration={snake,
								amplitude = 0.1mm,
								segment length = 7mm,
								post length=0.9mm},decorate](-1.96,2.8)--(0,-2.2);
							\draw[orange,thick,decoration={snake,
								amplitude = 0.1mm,
								segment length = 7mm,
								post length=0.9mm},decorate](0,-2.2)--(0.62,0.5);
							
							\draw[Maroon,thick,decoration={snake,
								amplitude = 0.05mm,
								segment length = 10mm,
								post length=0.9mm},decorate](0.62,0.5)--(1.7,6.5);
							\draw[Maroon,thick,decoration={snake,
								amplitude = 0.05mm,
								segment length = 10mm,
								post length=0.9mm},decorate](1.7,6.5)--(0,-1.7);
							\draw[Maroon,thick,decoration={snake,
								amplitude = 0.1mm,
								segment length = 7mm,
								post length=0.9mm},decorate](0,-1.7)--(-2.2,3.5);
							\draw[Maroon,thick,decoration={snake,
								amplitude = 0.1mm,
								segment length = 7mm,
								post length=0.9mm},decorate](-2.2,3.5)--(0,-5);
							\draw[Maroon,thick,decoration={snake,
								amplitude = 0.1mm,
								segment length = 7mm,
								post length=0.9mm},decorate](0,-5)--(3.2,4.5);
							\draw[Maroon,thick,decoration={snake,
								amplitude = 0.1mm,
								segment length = 7mm,
								post length=0.9mm},decorate](3.2,4.5)--(0,-3.5);
							\draw[Maroon,thick,decoration={snake,
								amplitude = 0.1mm,
								segment length = 7mm,
								post length=0.9mm},decorate](0,-3.5)--(-1.25,0.6);
							\draw[Maroon,thick,decoration={snake,
								amplitude = 0.1mm,
								segment length = 7mm,
								post length=0.9mm},decorate](-1.25,0.6)--(0,-3);
							\draw[Maroon,thick,decoration={snake,
								amplitude = 0.1mm,
								segment length = 7mm,
								post length=0.9mm},decorate](0,-3)--(1,0.44);
							
						\end{tikzpicture}
						
					}
					
					&
					
					\resizebox{0.5\textwidth}{!}{

						\begin{tikzpicture}[scale=1]
							\tikzstyle{every node}=[draw, circle, fill=white, minimum size=1.5pt, inner sep=0pt]
							\tikzstyle{dot}=[circle, fill=white, minimum size=0pt, inner sep=0pt, outer sep=-1pt]
							\node (n1) at (2*4/3,0) {};
							\node (n2) at (-2*4/3,2*2/3)  {};
							\node (n3) at (-2*2/3,-2*2/3)  {};
							
							\tikzstyle{every node}=[draw, circle, fill=white, minimum size=0.1pt, inner sep=0pt]
							\node (n6) at (-3,0) {};
							\node (n7) at (3,0)  {};
							\node (n8) at (0,1.5)  {};
							\node (n9) at (0,-1.5)  {};
							
							\foreach \from/\to in {n1/n2, n1/n3, n6/n7, n8/n9}
							\draw[decoration={snake,
								amplitude = 0.1mm,
								segment length = 7mm,
								post length=0.9mm},decorate] (\from) -- (\to);
							
							\draw[gray,dashed](n2)--(n3);
							
							\node[dot, draw=none, label=above: \tiny $Z$] at (2*4/3,0) {};
							\node[dot, draw=none, label=below: \tiny $Z^1$] at (-2*4/3,2*2/3) {};
							\node[dot, draw=none, label=above: \tiny $Z^2$] at (-2*2/3,-2*2/3) {};
							
							\node[dot, draw=none, label=above: \tiny $P_1$] at (-1,1.3) {};
							\node[dot, draw=none, label=below: \tiny $P_2$] at (0,-1.3) {};
							\node[dot, draw=none, label=above: \tiny $P_3$] at (1,0.8) {};
							\node[dot, draw=none, label=below: \tiny $C^1$] at (0,0.1) {};
							
							\node[dot, draw=none] at (0,-4.1) {};
							
							\draw[red,thick,decoration={snake,
								amplitude = 0.1mm,
								segment length = 7mm,
								post length=0.9mm},decorate](0.9,0.45)--(0,-0.89);
							\draw[red,thick,decoration={snake,
								amplitude = 0.1mm,
								segment length = 7mm,
								post length=0.9mm},decorate](-1.1,0.94)--(0,-0.89);
							
							\draw[Purple,decoration={snake,
								amplitude = 0.1mm,
								segment length = 7mm,
								post length=0.9mm},decorate](-1.19,0.98)--(-0.07,-0.9);	
							\draw[Purple,decoration={snake,
								amplitude = 0.1mm,
								segment length = 7mm,
								post length=0.9mm},decorate](-1.26,1)--(-0.12,-0.93);
							\draw[Purple,decoration={snake,
								amplitude = 0.1mm,
								segment length = 7mm,
								post length=0.9mm},decorate](-1.34,1.01)--(-0.18,-0.95);
							\draw[Purple,decoration={snake,
								amplitude = 0.1mm,
								segment length = 7mm,
								post length=0.9mm},decorate](-1.46,1.03)--(-0.26,-0.98);	
							
							\draw[orange,decoration={snake,
								amplitude = 0.1mm,
								segment length = 7mm,
								post length=0.9mm},decorate](0.82,0.45)--(0,-0.79);
							\draw[orange,decoration={snake,
								amplitude = 0.1mm,
								segment length = 7mm,
								post length=0.9mm},decorate](-1.02,0.91)--(0,-0.79);
							\draw[orange,decoration={snake,
								amplitude = 0.1mm,
								segment length = 7mm,
								post length=0.9mm},decorate](-0.94,0.89)--(0,-0.69);
							\draw[orange,decoration={snake,
								amplitude = 0.1mm,
								segment length = 7mm,
								post length=0.9mm},decorate](0.77,0.46)--(0,-0.69);
							\draw[orange,decoration={snake,
								amplitude = 0.1mm,
								segment length = 7mm,
								post length=0.9mm},decorate](0.72,0.49)--(0,-0.56);
							\draw[orange,decoration={snake,
								amplitude = 0.1mm,
								segment length = 7mm,
								post length=0.9mm},decorate](-0.85,0.87)--(0,-0.57);
							\draw[orange,decoration={snake,
								amplitude = 0.1mm,
								segment length = 7mm,
								post length=0.9mm},decorate](0.64,0.51)--(0,-0.43);
							\draw[orange,decoration={snake,
								amplitude = 0.1mm,
								segment length = 7mm,
								post length=0.9mm},decorate](-0.74,0.85)--(0,-0.43);
							
							\draw[Maroon,decoration={snake,
								amplitude = 0.1mm,
								segment length = 7mm,
								post length=0.9mm},decorate](0.95,0.44)--(0.1,-0.85);
							\draw[Maroon,decoration={snake,
								amplitude = 0.1mm,
								segment length = 7mm,
								post length=0.9mm},decorate](1,0.42)--(0.17,-0.83);
							\draw[Maroon,decoration={snake,
								amplitude = 0.1mm,
								segment length = 7mm,
								post length=0.9mm},decorate](1.06,0.41)--(0.24,-0.82);
							\draw[Maroon,decoration={snake,
								amplitude = 0.1mm,
								segment length = 7mm,
								post length=0.9mm},decorate](1.15,0.39)--(0.35,-0.78);
							
						\end{tikzpicture}
						
					}	
					
				\end{tabular}	
				
			\end{adjustbox}
			
			\caption{Left: $\partial^s T_0$ is red, $\partial^s T_1$ is red $\cup$ purple, $\partial^s T_2$ is red $\cup$ purple $\cup$ orange, and $\partial^s T_3$ is red $\cup$ purple $\cup$ orange $\cup$ maroon.\\ 
				Right: The red arc is $\gamma$ with $c(\gamma) = 3$, so $\Delta \setminus \{ \gamma \}$ has 3 components and one of them is $C^1$. Orange arcs in $C^1$ converge to $\gamma = [P_1, P_3]^s$ in the Hausdorff metric. Also, purple arcs converge to $[P_1, P_2]^s$, and maroon arcs converge to $[P_2, P_3]^s$, in the Hausdorff metric.}
			\label{fig:lozi2}
		\end{figure}
		
		Recall that $W^s_X = R \cup \left(\bigcup_{i=0}^{\infty}L^{-i}([X, V^0])\right)$, where $R$ is the component of $W^s_X\setminus\{X\}$ that is disjoint with $\Delta$, and $W^s_X \cap \Delta$ is dense
		in $\Delta$. Therefore $\left( \bigcup_{i=0}^{\infty}L^{-i}([X, V^0]) \right) \cap \Delta$ is 
		dense in $\Delta$.
		
		(1) Let $\gamma \in \Gamma$ be such
		that 
		$\gamma \setminus \partial\Delta$ is connected. 
		Let $\eps > 0$ be smaller than the
		distance between $\partial \gamma$ and the
		three vertices of $\Delta$, $Z$,
		$Z^1$ and $Z^2$, and let $U_{\eps}(\gamma) \subset \Delta$
		denote the open $\eps$-neighborhood
		of $\gamma$. Since 
		$\gamma \subset W^s_X$, there is an
		$N \in \N$ such that 
		$\gamma \subset \partial^s T_n$ for
		every $n \ge N$. Let us consider the
		maximal open connected surface $s^n$
		of $\Delta \cap T_n$ whose boundary
		contains $\gamma$. Let us denote $\partial^s s^n := \partial s^n \cap W^s_X$. There are 
		$\gamma_1, \dots , \gamma_{m(n)} \in \Gamma$, such that $\partial^s s^n = \bigcup_{i=1}^{m(n)} \gamma_i$ and
		$\gamma$ is one of them. Also
		$L^{-i}(V^0) \nin \Delta$ for
		every $i \in \N_0$, and hence
		either 
		$\gamma \subset L^{-n}([X, V^0])$,
		or $\gamma \subset L^{-n}([V^0, A^0])$.
		
		Since $\gamma$ separates $\Delta$ into exactly
		two components, $W^s_X$ is dense in $\B$,
		$L$ is order reversing on $W^u_X$ and the
		sequence of points $(A^{-n})_{n \in \N}$
		alternates around $X$ (and converges to
		$X$), we have
		$s^{n+2j} \subset s^{n+2(j-1)}$ and 
		$s^{n+2j} \cap s^{n+2j-1} = \emptyset$ for
		every $j \in \N$. Hence, there exists an
		$N' > N$ such that for every $n \ge N'$ we
		have $\partial^s s^{n} = \gamma \cup \gamma^n$ for some 
		$\gamma^n \in \Gamma$. Moreover, there
		exists an $n' \ge N'$ such that $\gamma^{n'}, \gamma^{n'+1} \subset U_{\eps}(\gamma)$
		and $\gamma^{n'}$, $\gamma^{n'+1}$ belong
		to different components for $\gamma$.
		
		(2) Let $c(\gamma) = n > 2$, 
		$\gamma \cap \partial \Delta = 
		\{ P_1, \dots , P_n \}$ and
		$P_1 \prec \dots \prec P_n$. Let 
		$k \in \{ 1, \dots , n-2 \}$ and consider
		the open component $C^k$ for 
		$[P_k, P_{k+2}] \subset \gamma$ that
		contains the whole arc $[P_k, P_{k+2}]$ in
		its boundary (see Figure \ref{fig:lozi2}, right). Let $\eps > 0$ be smaller
		than the distance between 
		$\{ P_k, P_{k+1}, P_{k+2} \}$ and
		$ \{ Z, Z^1, Z^2 \}$, and let
		$U_{\eps}([P_k, P_{k+2}]) \subset \Delta$
		denote the open $\eps$-neighborhood of
		$[P_k, P_{k+2}]$. Let $U_{\eps}^k := U_{\eps}([P_k, P_{k+2}]) \cap C^k$. Since
		$[P_k, P_{k+2}] \subset \gamma$,
		analogously as in (1), there is an $N \in \N$
		such that $[P_k, P_{k+2}] \subset
		\partial^s T_N$ and for the maximal open
		connected surface $s^N$ of
		$\Delta \cap T_N$ whose boundary contains
		$[P_k, P_{k+2}]$ we have
		$s^N \subset C^k$. Again, analogously as in
		(1), $s^{N+2j} \subset s^{N+2(j-1)}$ for
		every $j \in \N$, and hence there exists 
		an $n' > N$ such that $\partial^s s^{n'} = [P_k, P_{k+2}] \cup \gamma^{n'}$
		for some $\gamma^{n'} \in \Gamma$ and
		$\gamma^{n'} \subset U_{\eps}^k$. 
		
		The proof for the case when $[P_1, P_2] \subset \gamma$ or $[P_{n-1}, P_n] \subset \gamma$ follows analogously.
	\end{proof}
	
	To summarize, the arcs $\gamma \in\Gamma$
	satisfy the following properties:
	\begin{enumerate}[(A)]
		\item The elements of $\Gamma$ are pairwise disjoint or coincide (by the definition of $\Gamma$).
		\item Every $\gamma \in \Gamma$ is the limit of arcs in $\Gamma$ and is accumulated on both sides (by Lemma \ref{lem:gamma}).
		\item For $\gamma \in \Gamma$, the connected components of $L^{-1}(\gamma) \cap \Delta$ are elements of $\Gamma$ (since $L^{-1}(W^u_X \cap \partial \Delta) \subset W^u_X \cap \partial \Delta$ and the set of homoclinic points is invariant).
	\end{enumerate} 
	
	Now recall that a {\it metric tree} (or a {\it dendrite}) is a compact, connected, and locally connected metric space $\T$ (containing at least	two distinct points) such that for all $x, y \in \T$, $x \ne y$, there exists a unique arc in $\T$ with endpoints $x$ and $y$ (see e.g.\ \cite{BM}). We denote this unique arc in $\T$ by $[x, y]$.
	
	For completeness we recall the construction of the real tree $\T$ here and prove local connectedness of $\T$ in Lemma \ref{t:sc} below. For the proof of other properties of $\T$ the reader is referred to \cite{CP}. One denotes by $\Sigma$ the collection of sequences $(s_n)_{n=1}^\infty$ of open connected surfaces $s_n$ in $\Delta$ bounded by a finite number of elements in $\Gamma$, such that
	$\Cl(s_{n+1})\subset s_n$ for each $n$. Set $(s_n)_{n=1}^{\infty}\leq (s'_n)_{n=1}^{\infty}$ if for any $k$ there is $m$ such that 
	$\Cl(s_m)\subset s'_k$. Now let $\Sigma_0$ be the collection of sequences that are minimal for the relation $\leq$. One defines $\T$ as the quotient of $\Sigma_0$ by the relation "$\equiv$" defined by 
	$$(s_n)_{n=1}^{\infty}\equiv(s'_n)_{n=1}^\infty\textrm{ if and only if }(s_n)_{n=1}^{\infty}\leq(s'_n)_{n=1}^\infty\textrm{  and }(s'_n)_{n=1}^{\infty}\leq (s_n)_{n=1}^\infty.$$ 
	For convenience we now define a set-valued function, which was not originally in \cite{CP}. With $2^\Delta$ standing for the hyperspace of compact subsets of $\Delta$ we define $\psi : \T \to 2^\Delta$ by 
	$$\psi([(s_n)_{n=1}^{\infty}])=\bigcap_{n=1}^\infty s_n,$$ 
	which assigns to each equivalence class $[(s_n)_{n=1}^{\infty}]$ in $\T$ a continuum $\bigcap_{n=1}^\infty s_n$, which is a geometric realization of that class. It is easy to see that $\psi$ is one-to-one, since if $(s'_n)_{n=1}^{\infty}\notin [(s_n)_{n=1}^{\infty}]$ then $\bigcap_{n=1}^\infty s_n\neq\bigcap_{n=1}^\infty s'_n$. Let $\A=\psi(\T)$. Note that $\Gamma\subset\A$, but $\A\setminus \Gamma\neq\emptyset$. Indeed, an arc $\gamma\in\Gamma$ is represented by sequences $(s_n)_{n=1}^{\infty}$ such that $\gamma \subset s_n$ for each $n$ and $\bigcap s_n = \gamma$. However, there exist sequences $(s_n)_{n=1}^{\infty}$ such that $\bigcap s_n\notin\Gamma$, since $\Gamma$ is countable, but $\T$ is not. We shall call the elements of $\psi^{-1}(\Gamma)$ the {\it principal points} of $\T$, and the elements of $\psi^{-1}(\A\setminus \Gamma)$ the {\it ideal points} of $\T$. 
	
	We define $\tilde{\pi} : \mathcal{A} \to \T$ as the inverse of $\psi$; i.e.\ $\tilde{\pi} := \psi^{-1}$.
	
	\begin{lem}\label{t:sc} 
		Let $F$ be a Lozi-like map that satisfies $(L3)$ and $(L4)$. There exists a metric tree $\T$ and a continuous map $f : \T \to \T$ such that $(\Delta, F|_\Delta)$ is semi-conjugate to $(\T,f)$.
	\end{lem}
	\begin{proof}
		Since $\Gamma$ satisfies (A), (B), (C), the proof that $\T$ is a metric tree is the same as in the proof of \cite[Theorem 3]{CP}, with the only exception that it was not shown that $\T$ is locally connected, so we prove it in what follows. 
		
		Recall that the (countable) base in 
		$\T$ is determined as follows. Let $s'$ be a an open surface in $\Delta$ bounded by finitely many elements of $\Gamma$. The surface $s'$ gives an open set $U_{s'} = \{ (s_n)_{n \in \mathbb{N}} : \Cl(s_{n+1})\subset s_n, \forall n, \textrm{ and } \exists N, \, \Cl(s_n) \subset s', n \geq N \}$. We claim that $U_{s'}$ is connected. To justify this it is enough to show that given surfaces $s,t$ bounded by finitely many elements from $\Gamma$ we have $U_s\cap U_t=\emptyset$ if and only if $s\cap t=\emptyset$. First suppose $U_s\cap U_t=\emptyset$ but $s\cap t\neq\emptyset$. Since $s\cap t$ is open, and $W^s_X$ is dense in $\Delta$, by definition of $\Gamma$ there exists a $\gamma\in\Gamma$ such that $\gamma\cap s\cap t\neq\emptyset$. Since the elements of $\Gamma$ are pairwise disjoint it follows that $\gamma$ cannot intersect the boundaries of $s$ and $t$ and so $\gamma\subset s\cap t$. Since $\gamma$ is accumulated on both sides, there exists a sequence $\left(\gamma_n\right)_{n\in\mathbb{N}}$ converging to $\gamma$. Clearly $\gamma_n\in U_s\cap U_t$ for sufficiently large $n$, leading to a contradiction. Now suppose $s \cap t = \emptyset$. If $U_s \cap U_t \neq \emptyset$, there exists an $N$ and a decreasing sequence of surfaces $s_n$ such that $\Cl(s_n) \subset s\cap t$ for all $n\geq N$.
		Therefore, for any $\gamma \in \Gamma$ contained in the boundary of one of them, say \ $s_{N+1}$, we have $\gamma \subset s \cap t$. This contradiction completes the proof that any basic open set $U_{s'}$ is connected.
		
		The map $F$ induces a continuous map $f : \T \to \T$ in a natural way. Since $\bigcup\Gamma$ is dense in $\Delta$ the map $f$ is uniquely determined by letting $f(\tilde{\pi}(\gamma)) = \tilde{\pi}(F(\gamma))$ for each $\gamma \in \Gamma$.			
	\end{proof}
	
	Given a continuous map $f:X\to X$ on a compact metric space $X$, let
	\begin{equation}
		\underleftarrow{\lim} (X,f)
		:=
		\{\big(x_0,x_1,\ldots \big) \in X^{\mathbb{N}_0} :
		x_i\in X, x_i=f(x_{i+1}), \text{ for any }i \in \N_0 \}.
	\end{equation}
	$\underleftarrow{\lim} (X,f)$ is called the \emph{inverse limit of $X$ with the bonding map $f$}, or the \emph{inverse limit of $f$} for short. $\underleftarrow{\lim} (X,f) $ is equipped with
	metric induced from the 
	\emph{product metric} in $X^{\mathbb{N}_0}$. 
	The inverse limit space $\underleftarrow{\lim} (X,f)$
	also comes with a homeomorphism that extends $f$, 
	called the \emph{natural extension} of $f$, or the 
	\emph{shift homeomorphism}
	$\sigma_{f}:\underleftarrow{\lim}(X, f)
	\to \underleftarrow{\lim}(X, f)$, 
	defined as follows. 
	For any $\underbar x:= \big(x_0,x_1,x_2,\ldots \big)\in \underleftarrow{\lim}(X,f)$,
	\begin{equation}
		\sigma_f(\underbar x):=\big(f(x_0),f(x_1),f(x_2),\ldots \big)= \big(f(x_0),x_0,x_1,\ldots \big).
	\end{equation}
	It is well known that the natural extension $\sigma_f$ of a map $f$ is the `smallest' invertible map semi-conjugate to $f$ in the sense that any other invertible map semi-conjugate to $f$ is also semi-conjugate to $\sigma_f$.
	
	\begin{proof}[Proof of Theorem \ref{thm:conjugacy}.]
		First note that $F|_{\Lambda}$ is conjugate to the shift $\sigma_{F}$ on $\invlim(\Delta, F)$, since $\invlim(\Delta, F)$ is homeomorphic to $\Lambda$: If $(x_0, x_1, x_2, \dots) \in \invlim(\Delta, F)$, then $x_i \in \Delta$ for every $i \in \N_0$, implying $F^{-i}(x_0) \in \Delta$ for every $i \in \N$, and this holds only if $x_0 \in \Lambda$, so $h : \Lambda \to \invlim(\Delta, F)$ defined by $h(x_0) = (x_0, x_1, x_2, \dots)$, where $x_i = F^{-i}(x_0)$, is a homeomorphism.
		
		Let $\pi : \Delta \to \T$ be the projection map semi-conjugating $F|_{\Delta}$ to $f$, 
		guaranteed by Theorem \ref{t:sc}, i.e., $\pi(x)=t$ if $x\in\gamma\in\mathcal{A}$ and $\tilde{\pi}(\gamma)=t$. In particular $\pi$ identifies each $\gamma \in \Gamma$ with 
		a point in $\T$, $\pi(\Gamma)$ is dense in $\T$, and $\bigcup \Gamma$ is dense in $\Delta$. Let 
		$\Pi : \invlim(\Delta, F) \to \invlim(\T, f)$ be given by 
		$\Pi(x_0, x_1, x_2, \ldots) = (\pi(x_0), \pi(x_1), \pi(x_2), \ldots)$. We shall show that $\Pi$ 
		is injective. First we argue that $\Pi$ is injective on $\invlim(\bigcup \Gamma, F|_{\bigcup \Gamma})$. 
		Let $\Pi(x_0, x_1, x_2, \ldots) = \Pi(y_0, y_1, y_2, \ldots)$. Then $\pi(x_i) = \pi(y_i)$ for every 
		$i \in \N$. Denote by $\gamma_i$ the element of $\Gamma$ which contains $x_i$ and $y_i$, 
		$x_i, y_i \in \gamma_i$ for every $i \in \N$. Suppose by contradiction that there exists $k \in \N$ such that $x_k \neq y_k$. Then  
		$x_i \neq y_i$ for every $i \in \N$, since $F$ is a bijection. Note that for every arc $I \subset W^s_X$ there exists a point $P \in I$ which is not
		contained in $\Lambda$ and hence there is $K \in \N$ such that $F^{-K}(P) \nin \Delta$.
		Therefore, for $i \in \N$ there exist a point $P_i \in [x_i, y_i] \subset \gamma_i \subset W^s_X$
		and $K_i \in \N$ such that $F^{-K_i}(P_i) \nin \Delta$, contradicting the assumption that
		$[x_{i+K}, y_{i+K}] \subset \gamma_{i+K} \subset \Delta$.
		
		Therefore we have shown that $\Pi \circ \sigma_f(x) = \sigma_{F} \circ \Pi(x)$ for any 
		$x \in \invlim(\bigcup \Gamma, L|_{\bigcup \Gamma})$. Now suppose that $\Pi(x_0, x_1, x_2, \ldots) = \Pi(y_0, y_1, y_2, \ldots)$, and so $\pi(x_i) = \pi(y_i)$ for every 
		$i \in \N$, but $\pi^{-1}(x_i)\notin\Gamma$ for some, and consequently all, $i \in \N$. Let $\alpha_i=\pi^{-1}(x_i)$. 
	
		\begin{claim*}\label{Ci}
			$\alpha_i$ does not contain any subcontinuum $C$ of $\Lambda$
		\end{claim*}
	
		\begin{proof}[Proof of claim.]
			There exists a sequence of surfaces $(s^i_n)_{n=1}^{\infty }\in \Sigma$ such that $\alpha_i = \bigcap_{n\in\mathbb{N}} \Cl(s^i_n)$. Consequently $\alpha_i$ is a continuum. Note that 
			\begin{center}
				$I \not\subset \alpha_i \textrm{ for any arc } I \subset W^{u}_X,$
			\end{center}
			since then there would exist a sequence of arcs (contained in boundaries of surfaces $s^i_n$) $(I_n)_{n=1}^{\infty} \subset W^s_X$ such that $I_n\underset{n\to\infty}{\longrightarrow} I$ in the Hausdorff metric $d_H$, which is impossible by Definitions \ref{df:sh} and \ref{df:ll} that guarantee existence of disjoint stable and unstable cones. Similarly, $\alpha_i$ does not contain any subcontinuum $C$ of $\Lambda$, since then there would exist arcs $(I_n)_{n=1}^{\infty} \subset W^s_X$ and arcs $(J_n)_{n=1}^{\infty} \subset W^u_X$ such that $d_H(I_n, J_n)\underset{n \to \infty}{\longrightarrow} 0$.
		\end{proof}
	
		Since $\Delta \setminus \Lambda$ is dense and open in $\Delta$, so is $(\Delta \setminus \Lambda) \cap \alpha_i$ in $\alpha_i$ for all $i$. Hence $(\Delta \setminus \Lambda) \cap \alpha_i$ separates $x_i$ from $y_i$, and so again $\alpha_K \not\subset \Delta$ for some $K \in \mathbb{N}$. This contradiction concludes the proof. 
	\end{proof}

	\subsection{Conjugacy on the H\'enon attractors}
	
	Now we turn our attention to the H\'enon maps.
	Crovisier and Pujals proved in \cite[Theorems 2 and 3]{CP} that for $a \in (1, 2)$ and $b \in (-1/4, 1/4) \setminus \{ 0 \}$, the H\'enon map $H_{a,b}$ is mildly dissipative on the disc $\d = \{ (x, y) : |x| < 1/2 + 1/a, |y| < 1/2 - a/4 \}$, and therefore there exists a semi-conjugacy $\pi : (\d, H_{a,b}) \to (\T', h)$ to a continuous map $h$ on a compact real tree $\T'$ which induces an injective map on the set of non-atomic ergodic measures $\mu$ of  $H_{a,b}$, and such that the entropies of $\mu$ and $\pi_*(\mu)$ are the same.
	
	By choosing in their construction a particular family of arcs, we shall construct a metric tree $\T'$ and a map $h:\T'\to\T'$, and prove that $H_{a,b}|_{\Lambda}$ is conjugate to the shift homeomorphism $\sigma_f$ on $\invlim(\T',h)$ for $(a, b) \in \mathcal{WY}$, where $\mathcal{WY}$ is the Wang-Young parameter set from \cite[Corollary 1.3]{WY} for the orientation-reversing H\'enon maps $H_{a,b}$ ($b>0$). We shall do that by defining the set $\Gamma'$ for the H\'enon maps $H_{a,b}$ in an analogous way as the set $\Gamma$ is defined for the Lozi-like maps, so with stronger assumptions than in \cite{CP}, and show that the arcs of $\Gamma'$ satisfy analogs of (A), (B) and (C) when $(a, b) \in \WY$. So our choice of $\Gamma'$ is tailored for this particular set of parameters $\mathcal{WY}$. 
	
	 Let $D$ be a (closed) disc as in \cite[Lemma 4.4, Figure 2]{BK}, or as in \cite[Proposition 4.1, Figure 5]{MV} (see Figure \ref{fig:D} below). $D$ is mapped into itself by $H_{a,b}$, contains a fixed point $P$ of $H_{a,b}$ and $\Lambda = \Cl W^u_P \subset D$ by \cite[Proposition 4.1]{MV}. 
	 
	 \begin{figure}[h]
	 	\hspace{3cm}
	 	\resizebox{0.6\textwidth}{!}{
	 		\begin{tikzpicture}
	 			
	 			\fill[green!5!white] (-6.65,3.43)--(-4.74,-3.72) -- (-6.65,3.43) .. controls (10.75,0.9) and (10.75,-0.7) .. (-4.74,-3.75) -- cycle;
	 			
	 			\tikzstyle{every node}=[draw, circle, fill=white, minimum size=2pt, inner sep=0pt]
	 			\tikzstyle{dot}=[circle, fill=white, minimum size=0pt, inner sep=0pt, outer sep=-1pt]
	 			\node (n1) at (5*4/3,0) {};
	 			\node (n2) at (-6.65,5*2/3)  {};
	 			\node (n3) at (-4.74,-3.72)  {};
	 			\node (n4) at (5*5/9,0)  {};
	 			\node (n5) at (-4.59,-1.76) {};
	 			\node (n6) at (-3.98,-2.1)  {};
	 			\node (n7) at (-1.61,2.08)  {};
	 			\node (n8) at (5*61/51,0)  {};
	 			\node (n9) at (5*65/75,0)  {};
	 			
	 			\node (n10) at (5*4/9,1.87)  {};
	 			
	 			\node (n21) at (0.04,2.32)  {};
	 			\node (n22) at (-0.05,1.06)  {};
	 			\node (n23) at (0.05,-0.74)  {};
	 			\node (n24) at (-0.05,-1.28)  {};
	 			\node (n25) at (0.05,-2.08)  {};
	 			\node (n26) at (0,-2.74)  {};
	 			
	 			\node (n27) at (0,1.87)  {};
	 			\node (n28) at (-0.01,1.46)  {};
	 			
	 			\draw (-6.65,5*2/3)--(-4.74,-3.72);
	 			
	 			\draw (3,1.7) .. controls (7.7,0.7) and  (7.7,-0.6).. (4,-1.7);	
	 			\draw (3,1.7) .. controls (-9,4.25) and (-9,3.7) .. (-1.5,1.5);
	 			\draw (0,-2.1) .. controls (-6.9,-4.1) and (-6.9,-4.6) .. (4,-1.7);
	 			\draw (2.8,1.2) .. controls (7.44,0) and (7.44,-0.1) .. (0,-2.1);
	 			\draw (2.8,1.2) .. controls (-3,2.7) and (-3,2.2) .. (2.3,0.8);
	 			\draw (2.3,0.8) .. controls (5.63,-0.2) and (5.63,0.1) .. (-5*19/24+0.5,-5*5/12);
	 			\draw (-1.5,1.5) .. controls (4.55,-0.4) and (4.55,0.3) .. (-5*5/6+0.5,-5*1/3+0.15);
	 			\draw (-5*5/6+0.5,-5*1/3+0.15) .. controls (-5*5/6-0.7,-5*1/3-0.08) and (-5*5/6-0.7,-5*1/3-0.18) .. (-5*5/6+0.2,-5*1/3-0.05);
	 			\draw (-5*19/24+0.2,-5*5/12+0.1) .. controls (-5*19/24-0.15,-5*5/12) and (-5*19/24-0.15,-5*5/12-0.1) .. (-5*19/24+0.5,-5*5/12);
	 			
	 			\node[dot, draw=none, label=above: $P$] at (2.25,2.5) {};
	 			\node[dot, draw=none, label=above: $l$] at (-6,0) {};
	 			\node[dot, draw=none, label=above: $W^u_P$] at (5,2) {};
	 			\node[dot, draw=none, label=above: $D$] at (-3,0.5) {};
	 			\node[dot, draw=none, label=above: $\partial^uD$] at (5,-1.3) {};
	 			
	 		\end{tikzpicture}
	 	}
	 	\caption{Disc $D$. The boundary of $D$ consists of an arc of $W^u_P$, denoted by $\partial^u D$, and a straight line segment tangencial to $W^u_P$, denoted by $l$, so $\partial D = \partial^u D \cup l$.}
	 	\label{fig:D}
	 \end{figure}
	
	For arbitrary $b$ (positive or negative), small enough, it is shown in \cite{MV} that $W^s_P$ is dense in $D$. Therefore, if we let 
	
	$$\Gamma' = \{ \gamma : \gamma \textrm{ is a connected component of } D \cap W^s_P \}.$$
	then the following conditions hold:
	
	(A') The elements of $\Gamma'$ are pairwise disjoint or coincide (by the definition of $\Gamma'$).
	
	(B') Every $\gamma \in \Gamma'$ is the limit of arcs in $\Gamma'$ and is accumulated on both sides.
	
	(C') For $\gamma \in \Gamma'$ the connected components of $H_{a,b}^{-1}(\gamma) \cap D$ are elements of $\Gamma'$.

	By using (A'), (B'), (C') and $D$ as an analog of $\Delta$, the proof of Theorem \ref{cor:conjugacy} is practically the same as the proof of Theorem \ref{thm:conjugacy}, subject only to obvious modification of terminology and notation, with the exception that in order to justify Claim in the proof of Theorem \ref{thm:conjugacy} we use Theorem 1.2 in \cite{WY},that guarantees uniform hyperbolicity for all orbits that stay $\epsilon$ away from the critical set, for any $\epsilon>0$, instead of Definitions \ref{df:sh} and \ref{df:ll} used for the Lozi-like maps.

	\section{Geometry of the trees $\mathbb{T}$}\label{geo}
	
	In this section we will prove Theorems \ref{thm:dense} and \ref{cor:dense}. Again, we first focus on the Lozi-like maps. 
	
	\subsection{Dense branching for trees arising from the Lozi-like attractors}
	
	Let $F$ be a Lozi-like map that satisfies $(L3)$ and $(L4)$.  We will use the following notation: If $J \subset \R^2$ is an arc with endpoints $P$ and $Q$, then $\Int J = J \setminus \{ P, Q \}$. For $\alpha \in \A$ we denote by $c(\alpha)$ the number of components of
	$\Delta \setminus \{ \alpha \}$. By definition of branch points the following lemma holds.
	\begin{lem}
		$\pi(\alpha)$ is a branch point if and only if $c(\alpha) > 2$.
	\end{lem}
	
	Below we shall inductively define stems, levels of stems and levels of branch points. Let $\A^0 := \{ \alpha \in \A : \alpha \textrm{ separates $Z$ and $Z^1$} \} \cup \{ Z, Z^1 \}$. We call $B^0 := \pi(\A^0) \subset \T$ a {\it stem} of $\T$, and define the {\it level of the stem} $B^0$ to be zero. Let $E = \Delta \cap \kappa$. Recall that $Y^0 := Z^{-1}$ is an endpoint of $E$ (see Figure \ref{fig:lozi2}) and that the set of all basic V-points (that is, V-points of level 1) is dense in $E \subset \kappa$ (see Remark \ref{rem:dense}). Recall also that the `triangle' with vertices $X$, $Z$, $W$ (see Figure \ref{fig:D}) is mapped by $F^{-1}$ onto the `triangle' $T_0$.
	
	Suppose that $\alpha \in \A^0$ is such that one of its points lies in $(X, Z)^u$. Then, there is a point of $\alpha$ that lies in $(Z, W)^u$, and a point, say $Q \in \alpha \cap F(\kappa)$. Also, $F^{-1}(\alpha) \subset T_0$, there are two points of $F^{-1}(\alpha)$ that lie in $(A^0, X)$, $F^{-1}(Q) \in E$ is a V-point, and so $F^{-1}(\alpha)$ separates $Y^0$ and $\partial^s (T_0 \cap \Delta)$, but $F^{-1}(\alpha)$ need not be in $\A^0$. Since the set of basic V-points is dense in $E$, there exists a unique element of $\A^0$, say $\alpha^0 \in \A^0$, such that $\alpha^0$ separates $Y^0$ and $\partial^s(T_0 \cap \Delta)$. Note that since $\alpha^0 \in \A^0$, $\alpha^0$ also separates two components of $\partial^s(T_0 \cap \Delta)$ and hence $\alpha^0$ is the maximal element of $\A$ that separates $Y^0$ and $\partial^s(T_0 \cap \Delta)$ in the sense that every other element of $\A$ that separates $Y^0$ and $\partial^s(T_0 \cap \Delta)$ lies in the component of $\Delta \setminus \alpha^0$ that contains $Y^0$. Let $b^0 := \pi(\alpha^0)$. Note that $b^0$ is a branch point. We define the level of the branch point $b^0$ to be zero.
	
	\begin{figure}[ht]
		\begin{center}
			{\centering
				\begin{tikzpicture}[scale=1.3]
					\tikzstyle{every node}=[draw, circle, fill=white, minimum size=1.5pt, inner sep=0pt]
					\tikzstyle{dot}=[circle, fill=white, minimum size=0pt, inner sep=0pt, outer sep=-1pt]
					\node (n1) at (2*4/3,0) {};
					\node (n2) at (-2*4/3,2*2/3)  {};
					\node (n3) at (-2*2/3,-2*2/3)  {};
					\node (n4) at (2*4/9,2*2/9) {};
					\node (n5) at (0,0.67)  {};
					\node (n10) at (0,-2*4/3)  {};
					\node (n11) at (0,-0.89)  {};
					\node (n12) at (-4/3,1)  {};
					
					\node (n21) at (2*4/3+6,0) {};
					\node (n22) at (-2*4/3+6,2*2/3)  {};
					\node (n23) at (-2*2/3+6,-2*2/3)  {};
					\node (n24) at (2*4/9+6,2*2/9) {};
					\node (n25) at (0+6,0.67)  {};
					\node (n30) at (0+6,-2*4/3)  {};
					\node (n31) at (0+6,-0.89)  {};
					\node (n32) at (-4/3+6,1)  {};
					
					\tikzstyle{every node}=[draw, circle, fill=white, minimum size=0.1pt, inner sep=0pt]
					\node (n6) at (-2.5,0) {};
					\node (n7) at (3,0)  {};
					\node (n8) at (0,2)  {};
					\node (n9) at (0,-3)  {};
					
					\node (n26) at (-2.5+6,0) {};
					\node (n27) at (3+6,0)  {};
					\node (n28) at (0+6,2)  {};
					\node (n29) at (0+6,-3)  {};	
					
					\foreach \from/\to in {n1/n2, n1/n3, n2/n3, n6/n7, n8/n9, n4/n10, n10/n12}
					\draw (\from) -- (\to);
					
					\foreach \from/\to in {n21/n22, n21/n23, n22/n23, n26/n27, n28/n29, n24/n30, n30/n32}
					\draw (\from) -- (\to);
					
					\node[dot, draw=none, label=above: $Z$] at (2*4/3,0) {};
					\node[dot, draw=none, label=above: $Y^0$] at (0.2,1.2) {};
					\node[dot, draw=none, label=below: $A^0$] at (-1.3,1) {};
					\node[dot, draw=none, label=above: $L_{a,b}(Z)$] at (-2.6,2.15) {};
					\node[dot, draw=none, label=above: $L_{a,b}^2(Z)$] at (-1.25,-1.1) {};
					\node[dot, draw=none, label=below: $X$] at (8/9,2*2/9) {};
					\node[dot, draw=none, label=above: $V^0$] at (0.2,-2.6) {};
					\node[dot, draw=none, label=above: $Y^0_2$] at (0.2,-0.82) {};
					
					\node[dot, draw=none, label=above: $Z$] at (2*4/3+6,0) {};
					\node[dot, draw=none, label=above: $Y^0$] at (0.2+6,1.2) {};
					\node[dot, draw=none, label=below: $A^0$] at (-1.3+6,1) {};
					\node[dot, draw=none, label=above: $L_{a,b}(Z)$] at (-2.6+6,2.15) {};
					\node[dot, draw=none, label=above: $L_{a,b}^2(Z)$] at (-1.25+6,-1.1) {};
					\node[dot, draw=none, label=below: $X$] at (8/9+6,2*2/9) {};
					\node[dot, draw=none, label=above: $V^0$] at (0.2+6,-2.6) {};
					\node[dot, draw=none, label=above: $Y^0_2$] at (0.2+6,-0.82) {};
					
					\draw[Orange,thick](-4/3,1)--(-0.58,-1.08);
					\draw[Orange,thick](2*4/9,2*2/9)--(0.56,-0.71);
					
					\draw[Orange,thick](-4/3+6,1)--(-0.58+6,-1.08);
					\draw[Orange,thick](2*4/9+6,2*2/9)--(0.56+6,-0.71);
					
					\draw[red,thick](-0.52,0.8)--(0,-0.86);
					\draw[red,thick](0.5,0.54)--(0,-0.86);
					
					\draw[red,thick](-0.52+6,0.8)--(5.9,-0.92);
					\draw[red,thick](-0.52+6,0.8)--(0+6,-0.84);
					\draw[red,thick](0.5+6,0.54)--(0+6,-0.84);
					
					\draw[OliveGreen,thick](0,0.67)--(0,-0.9);
					
					\draw[OliveGreen,thick](0+6,0.67)--(0+6,-0.9);
					
			\end{tikzpicture}}
			
		\end{center}
		\caption{The region $T_0 \cap \Delta$. $\partial^s (T_0 \cap \Delta)$ is in orange, $E$ is in green, $\alpha^0$ is in red. $\alpha^0$ may have various ``forms''. In the left picture, the V-point in $\alpha^0 \cap \partial \Delta$ lies in $[Z, Z^2]^u$, and in the right picture the V-point in $\alpha^0 \cap \partial \Delta$ lies in $[Z, Z^1]^u$.}
		\label{fig:region}
	\end{figure}
	
	Let $\A^1 := \{ \alpha \in \A : \alpha
	\textrm{ separates $Y^0$ and $\alpha^0$} \} \cup \{ Y^0, \alpha^0 \}$. We call 
	$B^1 := \pi(\A^1) \subset \T$ a {\it stem} of $\T$, and define the {\it level of the
		stem} $B^1$ to be one. Note that $b^0 = B^0 \cap B^1$. 
	
	Let us assume that we have already defined $\A^n$,
	all stems of level $n$, and all branch
	points of level $n-1$. Now we shall define
	$\A^{n+1}$, stems of level $n+1$ and branch
	points of level $n$.
	
	Recall that $T_n = F^{-n}(T_0)$
	is a polygon and that $T_n \cap \Delta$
	consists of finitely many components
	$s^n_i$, $i = 1, \dots , m_n$, that is
	$T_n \cap \Delta = \bigcup_{i=1}^{m_n} s^n_i$, $s^n_i \cap s^n_j = \emptyset$
	for $i \ne j$. In addition, $F^{-n}(E)$ is an arc
	and $F^{-n}(E) \cap \partial \Delta$
	consists of finitely many points. Let
	$F^{-n}(E) \cap \partial \Delta = \{ Y^n_1, \dots , Y^n_{k_n}\}$, 
	where $F^{-n}(Y^0) = Y^n_1 \prec Y^n_2 \prec Y^n_3 \prec \cdots \prec Y^n_{k_n}$,
	and $E^n_j = [Y^n_{2j-1}, Y^n_{2j}] \subset \Delta$ is such that $E^n_j \cap \partial \Delta = \partial E^n_j$ for every 
	$j = 1, \dots , k_n/2$ (i.e.\ $\Int E^n_j \cap \partial \Delta = \emptyset$). Note that 
	$k_n/2 \ge m_n$, since it may happen that some $s^n_i$ contains more
	than one component $E^n_j$ of $F^{-n}(E) \cap \Delta$. 
	Let us fix $s^n_j$ for some $j \in \{ 1, \dots , m_n \}$. Recall that the stable boundary of $s^n_j$ is $\partial^s s^n_j = \partial s^n_j \cap W^s_X$ and the unstable boundary of $s^n_j$ is $\partial^u s^n_j = \partial s^n_j \cap \Delta$. There are finitely many $\gamma_1, \dots , \gamma_l \in \Gamma$ such that $\partial^s s^n_j = \bigcup_{i= 1}^l \gamma_i$. The unstable boundary $\partial^u s^n_j$ has the same number of components $\delta_1, \dots , \delta_l$, $\delta_i \subset \partial \Delta$ and $\partial^u s^n_j =  \bigcup_{i= 1}^l \delta_i$. 
	
	In order to define stems of level $n+1$ we have two cases to consider. First, let us suppose that there is a unique odd $p$, $1 \le p < k_n$, such that $Y^n_p, Y^n_{p+1} \in \partial s^n_j$. Since the set of V-points of level $n+1$ is dense in $E^n_{(p+1)/2} = [Y^n_p, Y^n_{p+1}]$, there is a unique element of $\bigcup_{q=0}^n \A^q$, say $\alpha^n_p \in$ $\bigcup_{q=0}^n \A^q$, such that $\alpha^n_p$ separates $Y^n_p$ and $\partial^s s^n_j$, and that is maximal with respect to this property, i.e., in the sense that every other element of $\A$ that separates $Y^n_p$ and $\partial^s s^n_j$ lies in the component of $\Delta \setminus \alpha^n_p$ that contains $Y^n_p$. 
	
	\begin{figure}[ht]
		\begin{center}
			{\centering
				\begin{tikzpicture}[scale=2.2]
					\tikzstyle{every node}=[draw, circle, fill=white, minimum size=1.5pt, inner sep=0pt]
					\tikzstyle{dot}=[circle, fill=white, minimum size=0pt, inner sep=0pt, outer sep=-1pt]
					\node (n1) at (2*4/3,0) {};
					\node (n2) at (-2*4/3,2*2/3)  {};
					\node (n3) at (-2*2/3,-2*2/3)  {};
					\node (n4) at (2*4/9,2*2/9) {};
					\node (n5) at (-0.36,0.75)  {};
					\node (n10) at (0,-2*4/3)  {};
					\node (n11) at (0.3,-0.79)  {};
					\node (n12) at (-4/3,1)  {};
					
					\tikzstyle{every node}=[draw, circle, fill=white, minimum size=0.1pt, inner sep=0pt]
					\node (n6) at (-3,0) {};
					\node (n7) at (3,0)  {};
					\node (n8) at (0,2)  {};
					\node (n9) at (0,-3)  {};
					
					\foreach \from/\to in {n1/n2, n1/n3, n2/n3, n6/n7, n8/n9, n4/n10, n10/n12}
					\draw (\from) -- (\to);
					
					\node[dot, draw=none, label=above: $Z$] at (2*4/3,0) {};
					\node[dot, draw=none, label=above: $Y^n_p$] at (-0.3,1.05) {};
					\node[dot, draw=none, label=below: $A^0$] at (-1.3,1) {};
					\node[dot, draw=none, label=above: $L_{a,b}(Z)$] at (-2.6,1.8) {};
					\node[dot, draw=none, label=above: $L_{a,b}^2(Z)$] at (-1.25,-1.2) {};
					\node[dot, draw=none, label=below: $X$] at (8/9,2*2/9) {};
					\node[dot, draw=none, label=above: $V^0$] at (0.2,-2.6) {};
					\node[dot, draw=none, label=above: $Y^n_{p+1}$] at (0.3,-0.8) {};
					
					\draw[Orange,thick](-0.15,0.71)--(0,-0.4);
					\draw[Orange,thick](0.41,0.57)--(0,-0.4);
					\draw[Orange,thick](0.57,0.53)--(0.37,-0.77);
					
					\draw[red,thick](-0.25,0.73)--(0,-0.6);
					\draw[red,thick](0.5,0.54)--(0,-0.6);
					\draw[red,thick](0.5,0.54)--(0.3,-0.78);
					
					\draw[red,thick](-0.5,0.79)--(0,-0.88);
					\draw[red,thick](0.4,0.2)--(0,-0.88);
					\draw[red,thick](0.4,0.2)--(0.3,-0.78);
					
					\draw[Orange,thick](-0.6,0.82)--(-0.1,-0.92);
					\draw[Orange,thick](0.33,-0.1)--(0.1,-0.85);
					\draw[Orange,thick](0.33,-0.1)--(0.25,-0.8);	
					
					\draw[OliveGreen,thick](-0.36,0.75)--(0,-0.75);
					\draw[OliveGreen,thick](0.44,0.34)--(0,-0.75);
					\draw[OliveGreen,thick](0.44,0.34)--(0.3,-0.78);
					
			\end{tikzpicture}}
			
		\end{center}
		\caption{The region $s^n_j$ such that $\partial^s s^n_j$ has 4 components (not in scale). $\partial^s s^n_j$ is in orange, $E^n_{(p+1)/2}$ is in green, $\alpha^n_p$ is in red. $\alpha^n_p$ may have various ``forms''. In this figure $\alpha^n_p \cap \partial \Delta$ contains 3 V-points. Some other possibilities are shown in the next figure.}
		\label{fig:region1}
	\end{figure}
	
	\begin{figure}[ht]
		\begin{center}{\centering
				\begin{tikzpicture}[scale=3.5]
					
					\tikzstyle{dot}=[circle, fill=white, minimum size=0pt, inner sep=0pt, outer sep=-1pt]
					
					\node (n10) at (2*4/9-2,2*2/9) {};
					\node (n11) at (0.56-2,-0.7)  {};
					\node (n12) at (-4/3-2,1)  {};
					\node (n13) at (-0.57-2,-1.1)  {};
					
					\foreach \from/\to in {n10/n12, n11/n13}
					\draw (\from) -- (\to);
					
					\node (n4) at (2*4/9,2*2/9) {};
					\node (n11) at (0.56,-0.7)  {};
					\node (n12) at (-4/3,1)  {};
					\node (n13) at (-0.57,-1.1)  {};
					
					\foreach \from/\to in {n4/n12, n11/n13}
					\draw (\from) -- (\to);	
					
					\node[dot, draw=none, label=above: $Y^n_p$] at (-0.3,1) {};
					\node[dot, draw=none, label=above: $Y^n_{p+1}$] at (0.3,-0.88) {};
					\node[dot, draw=none, label=above: $Y^n_p$] at (-0.3-2,1) {};
					\node[dot, draw=none, label=above: $Y^n_{p+1}$] at (0.3-2,-0.88) {};
					
					\tikzstyle{every node}=[draw, circle, fill=white, minimum size=1.5pt, inner sep=0pt]
					\tikzstyle{dot}=[circle, fill=white, minimum size=0pt, inner sep=0pt, outer sep=-1pt]		
					
					\node (n1) at (-0.3,0.74)  {};
					\node (n2) at (0.32,-0.78)  {};	
					\node (n3) at (-0.3-2,0.74)  {};
					\node (n4) at (0.32-2,-0.78)  {};	
					
					\draw[orange,thick](-0.15-2,0.7)--(0-2,-0.4);
					\draw[orange,thick](0.45-2,0.55)--(0-2,-0.4);
					\draw[orange,thick](0.57-2,0.53)--(0.37-2,-0.77);
					
					\draw[red,thick](-0.25-2,0.73)--(0-2,-0.59);
					\draw[red,thick](0.5-2,0.54)--(0-2,-0.59);
					\draw[red,thick](0.5-2,0.54)--(0.41-2,0.05);
					
					\draw[magenta,thick](0.36-2,0.05)--(0.05-2,-0.88);
					\draw[magenta,thick](0.54-2,0.53)--(0.32-2,-0.78);
					\draw[magenta,thick](0.36-2,0.05)--(0.32-2,-0.78);
					
					\draw[red,thick](-0.36-2,0.75)--(0-2,-0.89);
					\draw[red,thick](0.42-2,0.29)--(0-2,-0.89);
					\draw[red,thick](0.42-2,0.29)--(0.41-2,0.05);
					
					\draw[Orange,thick](-0.45-2,0.78)--(-0.07-2,-0.92);	
					\draw[orange,thick](0.33-2,-0.1)--(0.1-2,-0.86);
					\draw[orange,thick](0.33-2,-0.1)--(0.25-2,-0.81);
					
					\draw[OliveGreen,thick](-0.3-2,0.74)--(0-2,-0.74);
					\draw[OliveGreen,thick](0.45-2,0.4)--(0-2,-0.74);
					\draw[OliveGreen,thick](0.45-2,0.4)--(0.32-2,-0.78);

					\draw[orange,thick](-0.15,0.71)--(0,-0.38);
					\draw[orange,thick](0.42,0.56)--(0,-0.38);
					\draw[orange,thick](0.65,0.5)--(0.43,-0.74);
					
					\draw[red,thick](-0.25,0.73)--(0,-0.5);
					\draw[red,thick](0.5,0.54)--(0,-0.5);
					\draw[red,thick](0.5,0.54)--(0.41,0.05);
					
					\draw[magenta,thick](-0.5,0.79)--(0,-0.89);
					\draw[magenta,thick](0.37,0.1)--(0,-0.89);
					\draw[magenta,thick](0.55,0.53)--(0.395,-0.1);
					\draw[magenta,thick](0.37,0.1)--(0.395,-0.1);
					
					\draw[Mulberry,thick](0.33,-0.1)--(0.05,-0.88);
					\draw[Mulberry,thick](0.6,0.52)--(0.32,-0.78);
					\draw[Mulberry,thick](0.33,-0.1)--(0.32,-0.78);
					
					\draw[red,thick](-0.37,0.76)--(0,-0.7);
					\draw[red,thick](0.42,0.29)--(0,-0.7);
					\draw[red,thick](0.42,0.29)--(0.41,0.05);
					
					\draw[orange,thick](-0.58,0.81)--(-0.07,-0.93);
					\draw[orange,thick](0.29,-0.3)--(0.12,-0.86);
					\draw[orange,thick](0.29,-0.3)--(0.23,-0.82);	
					
					\draw[OliveGreen,thick](-0.3,0.74)--(0,-0.6);
					\draw[OliveGreen,thick](0.45,0.4)--(0,-0.6);
					\draw[OliveGreen,thick](0.45,0.4)--(0.32,-0.78);
					
			\end{tikzpicture}}
			
		\end{center}
		\caption{The region $s^n_j$ as in the previous figure (not in scale). $\partial^s s^n_j$ is in orange, $E^n_{(p+1)/2}$ is in green, $\alpha^n_p$ is in red.  In the left picture $\alpha^n_p \cap \partial \Delta$ contains 2 V-points, and in the right picture $\alpha^n_p \cap \partial \Delta$ contains 1 V-point.}
		\label{fig:region2}
	\end{figure}
	
	\begin{figure}[ht]
		\begin{center}{\centering
				\begin{tikzpicture}[scale=3.5]
					
					\tikzstyle{dot}=[circle, fill=white, minimum size=0pt, inner sep=0pt, outer sep=-1pt]
					
					\node (n10) at (2*4/9,2*2/9) {};
					\node (n11) at (0.56,-0.7)  {};
					\node (n12) at (-4/3,1)  {};
					\node (n13) at (-0.57,-1.1)  {};
					
					\foreach \from/\to in {n10/n12, n11/n13}
					\draw (\from) -- (\to);	
					
					\node[dot, draw=none, label=above: $Y^n_p$] at (-0.45,1) {};
					\node[dot, draw=none, label=above: $Y^n_{p+1}$] at (-0.05,-0.95) {};
					\node[dot, draw=none, label=above: $Y^n_{p+2}$] at (0.2,-0.88) {};
					\node[dot, draw=none, label=above: $Y^n_{p+3}$] at (0.44,0.78) {};
					\node[dot, draw=none, label=above: $Y^n_{p+4}$] at (0.7,0.72) {};
					\node[dot, draw=none, label=above: $Y^n_{p+5}$] at (0.45,-0.78) {};
					
					\tikzstyle{every node}=[draw, circle, fill=white, minimum size=1.5pt, inner sep=0pt]
					\tikzstyle{dot}=[circle, fill=white, minimum size=0pt, inner sep=0pt, outer sep=-1pt]		
					
					\node (n1) at (-0.45,0.78)  {};
					\node (n2) at (-0.02,-0.91)  {};	
					\node (n3) at (0.03,-0.89)  {};
					\node (n4) at (0.4,0.57)  {};	
					\node (n5) at (0.56,0.53)  {};
					\node (n6) at (0.34,-0.78)  {};	
					
					\draw[orange,thick](-0.15,0.71)--(0,-0.38);
					\draw[orange,thick](0.36,0.58)--(0,-0.38);
					\draw[orange,thick](0.65,0.5)--(0.43,-0.74);
					
					\draw[Magenta,thick](-0.25,0.73)--(0,-0.7);
					\draw[Magenta,thick](0.4,0.56)--(0,-0.7);
					\draw[Magenta,thick](0.4,0.56)--(0.09,-0.87);
					
					\draw[Mulberry,thick](0.45,0.4)--(0.14,-0.85);
					\draw[Mulberry,thick](0.62,0.51)--(0.34,-0.78);
					\draw[Mulberry,thick](0.45,0.4)--(0.34,-0.78);
					
					\draw[red,thick](-0.37,0.76)--(-0.02,-0.9);
					\draw[red,thick](-0.5,0.79)--(-0.02,-0.9);
					
					\draw[orange,thick](-0.58,0.81)--(-0.07,-0.93);
					\draw[orange,thick](0.34,-0.2)--(0.25,-0.81);
					\draw[orange,thick](0.34,-0.2)--(0.2,-0.83);	
					
					\draw[OliveGreen,thick](-0.45,0.78)--(-0.02,-0.9);
					\draw[OliveGreen,thick](0.4,0.56)--(0.03,-0.89);
					\draw[OliveGreen,thick](0.56,0.52)--(0.34,-0.78);
					
			\end{tikzpicture}}
			
		\end{center}
		\caption{The region $s^n_j$ (not in scale). $\partial^s s^n_j$ is in orange, $E^n_{(p+1)/2}$, $E^n_{(p+3)/2}$ and $E^n_{(p+5)/2}$ are in green, $\alpha^n_p$ is in red, $\alpha^n_{p+2}$ is in magenta and $\alpha^n_{p+4}$ is in purple.}
		\label{fig:region3}
	\end{figure}
	Let $\A^{n+1}_p := \{ \alpha \in \A : \alpha \textrm{ separates $Y^n_p$ and $\alpha^n_p$} \} \cup \{ Y^n_p, \alpha^n_p \}$. We call $B^{n+1}_p := \pi(\A^{n+1}_p) \subset \T$ a {\it stem} of $\T$ of {\it level} $n+1$. Let $b^n_p := \pi(\alpha^n_p)$, $b^n_p$ is a branch point and we define its {\it level} to be $n$.
	
	Second, let us suppose that $F^{-n}(E) \cap \partial s^n_j = \{ Y^n_p, \dots , Y^n_{p+r} \}$ and $r > 1$. Note that $(r+1)/2 \le l$. Since the set of V-points of level $n+1$ is dense in each $E^n_{(p+1+2i)/2}$, $i = 0, \dots , (r-1)/2$, there are elements of  $\bigcup_{q=0}^n \A^q$, say $\alpha^n_p, \alpha^n_{p+2}, \dots \alpha^n_{p+r-1} \in$  $\bigcup_{q=0}^n \A^q$, such that for every $k$, $0 \le k \le (r-1)/2$, $\bigcup_{i=0}^{k} \alpha^n_{p+2i}$ separates $Y^n_{p+2k}$ and $\partial^s s^n_j$, and for every $i$, $0 \le i \le (r-1)/2$, $\alpha^n_{p+2i}$ is maximal with respect to this property (see Figure \ref{fig:region3}).
	
	Obviously, $\pi(\alpha^n_{p+2i})$ is a branch point for every $i$. If additionally for some $i$, $\alpha^n_{p+2i}$ separates $Y^n_{p+2i}$ and $\partial^s s^n_j$, then
	$\pi(\alpha^n_{p+2i})$ is a branch point of {\it level} $n$. If $\alpha^n_{p+2i}$ does not separate $Y^n_{p+2i}$ and $\partial^s s^n_j$, then the branch point $\pi(\alpha^n_{p+2i})$ has level $m$ for some $m < n$. Note that $\alpha^n_p$ separates $Y^n_p$ and $\partial^s s^n_j$, and hence $\pi(\alpha^n_p)$ is a branch point of level $n$.
	
	For every $i \in \{  1, \dots , (r-2)/2 \}$ such that $\pi(\alpha^n_{p+2i})$ is a branch point of level $n$ let $\A^{n+1}_{p+2i} := \{ \alpha \in \A : \alpha \textrm{ separates $Y^n_{p+2i}$ and $\alpha^n_{p+2i}$} \} \cup \{ Y^n_{p+2i}, \alpha^n_{p+2i} \}$. We call $B^{n+1}_{p+2i} := \pi(\A^n_{p+2i}) \subset \T$ a {\it stem} of $\T$ of {\it level} $n+1$. 
	
	Note that, by our construction, between any two branch points of level $n \ge 1$ there is at least one branch point of some smaller level. Namely, the stable boundary of every $s^n_j \subset T_n$ contains at least one component which belongs to $\partial^s T_{n-1}$.
	
	Finally, let $\A^{n+1} = \bigcup_i \A^{n+1}_i$, where $i \in \{  1, \dots , (r-2)/2 \}$ such that $\pi(\alpha^n_{p+2i})$ is a branch point of level $n$.
	
	Above we have proved the following lemma. 
	\begin{lem}\label{lem:stem}
		For every component $s^n$ of $T_n \cap \Delta$ there exists 
		$\alpha^n \in \A$ such that $\alpha^n\subset s^n$ and 
		$\pi(\alpha^n)$ is a branch point of level $n$ of $\T$.
	\end{lem}
	
	
	\begin{proofof} \emph{of Theorem \ref{thm:dense}.}
		Let $\gamma \in \Gamma$. Let us suppose that
		$\pi(\gamma)$ is not a branch point of $\T$.
		Let $\eps > 0$ and let $U_{\eps}$ be the
		open $\eps$-neighborhood of $\gamma$ in
		$\Delta$. By Lemma \ref{lem:gamma}, there
		exists $n \in \N$ and a component $s^n$ of 
		$T_n \cap \Delta$ such that $\gamma$ is a
		component of $\partial^s s^n$ and 
		$s^n \subset U_{\eps}$. By Lemma \ref{lem:stem} there exists an
		$\alpha^n \in \A$ such that $\alpha^n\subset s^n$ and 
		$\pi(\alpha^n)$ is a branch point of $\T$.
		Since $\Gamma$ is a dense subset of $\A$, the
		set of branched points is dense in $\T$.
	\end{proofof}
	
	\subsection{Dense branching for trees arising from the H\'enon attractors}	

	The analogues of points of $W^u_X \cap E$ of the Lozi-like maps are critical points of the H\'enon maps. The set of critical points $\mathcal{C}$ is introduced and studied in detail in \cite{WY}. It is shown in \cite{BS2} that $\mathcal{C}$ lies on an arc $\K$ (constructed in \cite{BS2}), and for the H\'enon maps the arc $\K$ plays the role of the arc $E = \Delta \cap \kappa$. Note that all points of $H_{a,b}^{-n}(\K) \cap \partial^u D$, $n \in \N$, are preimages of some critical points and these points are the analogues of points $Y^n_q$.
	
	Let $\Omega$ be a region defined similarly as in \cite[Proposition 4.2., Figure 6]{MV}, we indicate that region in Figure \ref{fig:omega} below, where it is bounded by the red and blue arcs. For us $\Omega$ is a (closed) disc (not open as in \cite{MV}) such that $\partial \Omega \subset W^u_P \cup W^s_P$. Also, if we denote by  $\gamma_P \in \Gamma'$ the arc of $W^s_P$ that has $P$ as a boundary point, and if $W'$ denotes the other boundary point of $\gamma_P$, that is $\gamma_P = [X, W']^s$, then $\partial^u \Omega := [P, H_{a,b}^{-1}(W')]^u$ and $\partial^s \Omega := [P, H_{a,b}^{-1}(W')]^s$, each of them has only one component (not two as in \cite{MV}). Therefore $\Omega$ is not contained in $D$ as in \cite{MV}, but conclusion of \cite[Proposition 4.2]{MV} holds.
	
		\begin{figure}[h]
		
		\begin{adjustbox}{center}
			\resizebox{0.6\textwidth}{!}{
				
				\begin{tikzpicture}
					
					\tikzstyle{every node}=[draw, circle, fill=white, minimum size=2pt, inner sep=0pt]
					\tikzstyle{dot}=[circle, fill=white, minimum size=0pt, inner sep=0pt, outer sep=-1pt]
					
					\node (n1) at (5*4/3,0) {};
					\node (n2) at (-6.65,5*2/3)  {};
					\node (n3) at (-4.74,-3.72)  {};
					\node (n4) at (5*5/9,0)  {};
					\node (n5) at (-4.59,-1.76) {};
					\node (n6) at (-3.98,-2.1)  {};
					\node (n7) at (-4.61,2.84)  {};
					\node (n8) at (5*61/51,0)  {};
					\node (n9) at (5*65/75,0)  {};
					
					\node (n10) at (5*4/9,1.87)  {};
					
					\node (n21) at (0.04,2.32)  {};
					\node (n22) at (-0.05,1.06)  {};
					\node (n23) at (0.05,-0.74)  {};
					\node (n24) at (-0.05,-1.28)  {};
					\node (n25) at (0.05,-2.08)  {};
					\node (n26) at (0,-2.74)  {};
					
					\node (n27) at (0,1.89)  {};
					\node (n28) at (-0.01,1.41)  {};
					
					\node (n11) at (-2.06,2.74)  {};
					\node (n12) at (1.4,-2.39)  {};
					
					\draw (-6.65,5*2/3)--(-4.74,-3.72);
					
					\draw[TealBlue,thick](-0.05,1.06)--(0.05,-0.74);
					\draw[TealBlue,thick](0.05,-0.74)--(-0.05,-1.25);
					\draw[TealBlue,thick](-0.05,-1.25)--(0.05,-2.08);
					\draw[TealBlue,thick](n25)--(n26);
					\draw[TealBlue,thick](n27)--(n28);
					\draw[TealBlue,thick](n21)--(n27);
					\draw[TealBlue,thick](n28)--(n22);
					
					\draw (3,1.7) .. controls (7.7,0.7) and  (7.7,-0.6).. (4,-1.7);	
					\draw (3,1.7) .. controls (-9,4.25) and (-9,3.7) .. (-1.5,1.5);
					\draw (0,-2.1) .. controls (-6.9,-4.1) and (-6.9,-4.6) .. (4,-1.7);
					\draw (2.8,1.2) .. controls (7.44,0) and (7.44,-0.1) .. (0,-2.1);
					\draw (2.8,1.2) .. controls (-7,3.7) and (-7,3.2) .. (2.3,0.8);
					\draw (2.3,0.8) .. controls (5.63,-0.2) and (5.63,0.1) .. (-5*19/24+0.5,-5*5/12);
					\draw (-1.5,1.5) .. controls (4.55,-0.4) and (4.55,0.3) .. (-5*5/6+0.5,-5*1/3+0.15);
					\draw (-5*5/6+0.5,-5*1/3+0.15) .. controls (-5*5/6-0.7,-5*1/3-0.08) and (-5*5/6-0.7,-5*1/3-0.18) .. (-5*5/6+0.2,-5*1/3-0.05);
					\draw (-5*19/24+0.2,-5*5/12+0.1) .. controls (-5*19/24-0.15,-5*5/12) and (-5*19/24-0.15,-5*5/12-0.1) .. (-5*19/24+0.5,-5*5/12);
					
					\draw[red,thick] (5*4/9,1.87) .. controls (0.7,-7.7) and  (-0.7,-7.6) .. (-5*4/9,3.87);	
					
					\draw[blue,thick] (5*4/9,1.87) -- (-2.06,2.74);
					
					\node[dot, draw=none, label=above: $P$] at (2.25,2.5) {};
					\node[dot, draw=none, label=above: $W^u_P$] at (5,2) {};
					\node[dot, draw=none, label=above: $\Omega$] at (0.1,-3.6) {};
					\node[dot, draw=none, label=above: $\K$] at (0.25,0.5) {};
					\node[dot, draw=none, label=above: $\gamma_P$] at (1.65,0.5) {};
					\node[dot, draw=none, label=above: $W^s_P$] at (-2.7,4) {};
					\node[dot, draw=none, label=above: $H^{-1}_{a,b}(W')$] at (-1.1,4)  {};
					\node[dot, draw=none, label=above: $W'$] at (1.55,-2.39)  {};
					
				\end{tikzpicture}
			}
		\end{adjustbox}
		\vspace{-2cm}
		\caption{The region $\Omega$ with $\partial^u \Omega$ in blue and $\partial^s \Omega$ in red. The arc $\K$ is teal.}
		\label{fig:omega}
	\end{figure}
	
	\begin{proof}[Proof of Theorem \ref{cor:dense}.]  By using  (A'), (B') and (C'), $D$ as an analog of $\Delta$, $\Omega$ as an analog of $T_0$, and the critical points and their preimages as the analog of points $Y^n_q$, the proof of Theorem \ref{cor:dense} is practically the same as the proofs of Lemma \ref{lem:stem} and Theorem \ref{thm:dense}, subject only to obvious modification of terminology and notation.
	\end{proof}

		\subsection{Densely branching trees are optimal}
		
In this subsection we show that our results are optimal by proving Theorem \ref{thm:LO}. 

\begin{proof}[Proof of Theorem \ref{thm:LO}]
	Let $F$ be a Lozi-like map that satisfies $(L3)$ and $(L4)$. Note that $Per(F)$, the set of periodic points of $F$, is dense in $\Lambda_F$ (due to density of homoclinic intersections). Let us suppose, by contradiction, that there exists a tree $\T$ whose set of branch points is not dense, and $f : \T \to \T$ such that $(\invlim (\T, f), \hat f)$ is conjugate to $(\Lambda_F, F)$. Let also $\pi:\Lambda_F\to\T$ be the semi-conjugacy for $f$. Then $\pi(Per(F))\subset Per(f)$ and $\pi(Per(F))$ is dense in $\T$. Note that given any $\alpha\in\A$ the image $\pi(\alpha)$ must be a point in $\T$. Indeed, if $x,y\in\alpha$ were such that $\pi(x)\neq \pi(y)$ then $\pi(\alpha)$ is a nondegenerate subcontinuum of $\T$, and hence it contains an open set. Therefore there exist points $z,z'\in\alpha$ and $t_o,t_o'\in \pi(Per(F))$ such that $\pi(z)=t_o$ and  $\pi(z')=t_o'$ are distinct periodic points of $f$. Let $\epsilon=\min\{d(f^n(t_o),f^n(t_o')):n\in\mathbb{N}\}$ and $\delta>0$ be chosen from continuity of $\pi$ for $\epsilon/10$. Since $\diam (F^n(\alpha))\to_{n\to\infty}0$ it follows that there exists an $n$ such that $d(F^n(z),F^n(z'))<\delta$. Consequently $d(\pi(F^n(z)),\pi(F^n(z'))=d(f^n(t_o),f^n(t_o'))<\epsilon/10$; a contradiction. 
	
	By the above, and since $\T$ contains an arc with finitely many branch points, there exist a point $t \in \T$ and an infinite collection $A = \{ \alpha_i, i \in \N \} \subset \A$, $\alpha_i \ne \alpha_j$ for every $i \ne j$, such that $\pi(\alpha_i) = t$ for every $i \in \N$.
	
	Let $j, j' \in \N$, $j \ne j'$, $Q \in \alpha_{j}$, $Q' \in \alpha_{j'}$, and let $\bar q, \bar q' \in \invlim (\T, f)$ be such that $\Pi (Q) = \bar q$ and $\Pi (Q') = \bar q'$. Since $\pi(Q) = \pi(Q') = t$, it follows that $q_0 = q'_0 = t$ and consequently $f^i(q_0) = f^i(q'_0) = f^i(t)$ for every $i \in \N$. Therefore, for every open cover $\U$ of $\invlim (\T, f)$, there exists a $k \in \N$ such that for every $i \ge k$ we have $\hat f^i(\bar q), \hat f^i(\bar q') \in U$ for some $U \in \U$. Since $\Pi$ is a conjugacy, it follows that $d(Q^i, Q'^i) \to 0$ as $i \to \infty$, that is $Q' \in W^s_{Q}$, and consequently $\alpha _i \subset W^s_{Q}$ for every $i \in \N$. We have the following three cases. 
	
	{\bf Case 1.} First, let us suppose that $W^s_Q = W^s_X$ and let $\alpha_X \in \A$ be such that $X \in \alpha_X$. Recall that $\pi(\alpha_X) = x$ and $x$ is the only fixed point of $f$. Since for any $\alpha \in \A$ there are finitely many $\beta_i \in \A$ such that for every $i$, $F(\beta_i) \subset \alpha$, there exist a sequence $(\alpha_{n_i})_i$, $\alpha_{n_i} \in A$, and a strictly increasing sequence $(r_i)_i$, $r_i \in \N$, such that $F^{r_i}(\alpha_{n_i}) \subset \alpha_X$ and $F^{r_i-1}(\alpha_{n_i}) \cap \alpha_X = \emptyset$. In fact, if $J = \alpha_X \, \cap(\mathbb{R}\times(-\infty,0])$ then $J$ is a closed arc and  $F^{r_i}(\alpha_{n_i}) \subset J$, for every $i \in \N$. Since for every $i > 1$, $r_i > r_1$, $\pi(\alpha_{n_i}) = t$ and $f^{r_1}(t) = x$, for any point $P \in F^{r_i}(\alpha_{n_i}) \subset J$ and $\Pi(P) = \bar p = (p_0, p_1, \dots , p_{r_i}, \dots )$ we have $p_{r_i} = t$, $p_{r_i - r_1} = f^{r_1}(t) = x$ and hence $p_j = x$ for every $j = 0, \dots , r_i - r_1$. Recall that $r_i \to \infty$ when $i \to \infty$, and $J$ is a closed arc, so there exists a point $P' \in J$ such that $\Pi(P') = (p'_0, p'_1, \dots , p'_j, \dots )$ and $p'_j = x$ for every $j \in \N$. This is in contradiction with the fact that $\Pi$ is a conjugacy and $\Pi(X) = (x, x, \dots , x, \dots )$. Therefore, $W^s_Q \ne W^s_X$.
	
	{\bf Case 2.} In the same way, using $F^n$ instead of $F$, we can prove that $W^s_Q \ne W^s_P$ for every periodic point $P$ of period $n$, and for every $n$. Therefore, $W^s_Q \ne W^s_{Q^{-j}} \ne W^s_{Q^{-k}}$ for every $j, k \in \N$, $j \ne k$.
	
	{\bf Case 3.} Now let us suppose that there does not exist a periodic point $R$, such that $Q\in W^s_R$. Let us fix $t_1 \in f^{-1}(t)$. Then for every $\alpha_i \in A$ there exists $\alpha_i^1 \in  \A$, $\alpha_i^1 \subset W^s_{Q^{-1}}$, such that $F(\alpha_i^1) \subset \alpha_i$ and $\pi(\alpha_i^1) = t_1$. Now we fix $t_2 \in f^{-1}(t_1) \subset f^{-2}(t)$, and again for every $\alpha_i^1$ there exists $\alpha_i^2 \in  \A$, $\alpha_i^2 \subset W^s_{Q^{-2}}$, such that $F(\alpha_i^2) \subset \alpha_i^1$ and $\pi(\alpha_i^2) = t_2$. Note that $F^2(\alpha_i^2) \subset \alpha_i$. We proceed inductively and fix a sequence $(t_m)_m$ such that for every $m \in \N$, $t_m \in f^{-1}(t_{m-1}) \subset f^{-m}(t)$. Now for every $i \in \N$ there exists a sequence $(\alpha_i^m)_m$ such that for every $m \in \N$ the following holds: $\alpha_i^m \in \A$, $\alpha_i^m \subset W^s_{Q^{-m}}$,  $F(\alpha_i^m) \subset \alpha_i^{m-1}$ and $\pi(\alpha_i^m) = t_m$. Note that $F^m(\alpha_i^m) \subset \alpha_i$ for every $m \in \N$. Also, $\alpha_i \cap \alpha_j = \emptyset$ for $i \ne j$ implies that $\alpha_i^m \cap \alpha_j^m = \emptyset$ for $i \ne j$ and every $m \in \N$. Since every $\alpha \in \A$ is compact, there exists $P_i \in \bigcap_{m \in \N} F^m(\alpha_i^m) \subset \alpha_i$ for every $i \in \N$. Note that $P_i \in \Lambda_F$ for every $i \in \N$. Moreover, $P_i \ne P_j$ for $i \ne j$ and $\Pi(P_i) = (t, t_1, t_2, \dots , t_m, \dots )=\Pi(P_j)$ for every $i,j \in \N$, contradicting the fact that $\Pi$ is a conjugacy. Therefore, there exists no tree $\T$ whose set of branch points is not dense and $f : \T \to \T$ such that $(\invlim (\T, f), \hat f)$ is conjugate to $(\Lambda_F, F)$.
\end{proof}
The proof also works verbatim for the H\'enon-like maps.

	\subsection{Remarks}	
	
	From everything above, a few very interesting questions arise. Recall, a branch point $b \in \T$ has \emph{order} $n$ (where $n \ge 3$) if $\T \setminus \{ b \}$ consists of $n$ components.
	
	Are there parameters $(a,b) \in \M$ such that all branch point of the metric tree $\T$ for the Lozi map $L_{a,b}$ have order three? If yes, is there a metric $d$ on $\T$ (related to the dynamics of $L_{a,b}$) such that $(\T, d)$ is the \emph{continuum self-similar tree} (see \cite{BM} and \cite{BT}).
	
	Similarly, are there parameters $(a,b) \in \M$ such that the metric tree $\T$ for the Lozi map $L_{a,b}$ has at least one branch point of order greater than three? In this case, is there an upper bound on the order of branch points?
	
	\begin{figure}[ht]
		\centering		\includegraphics[width=15cm,height=15cm]{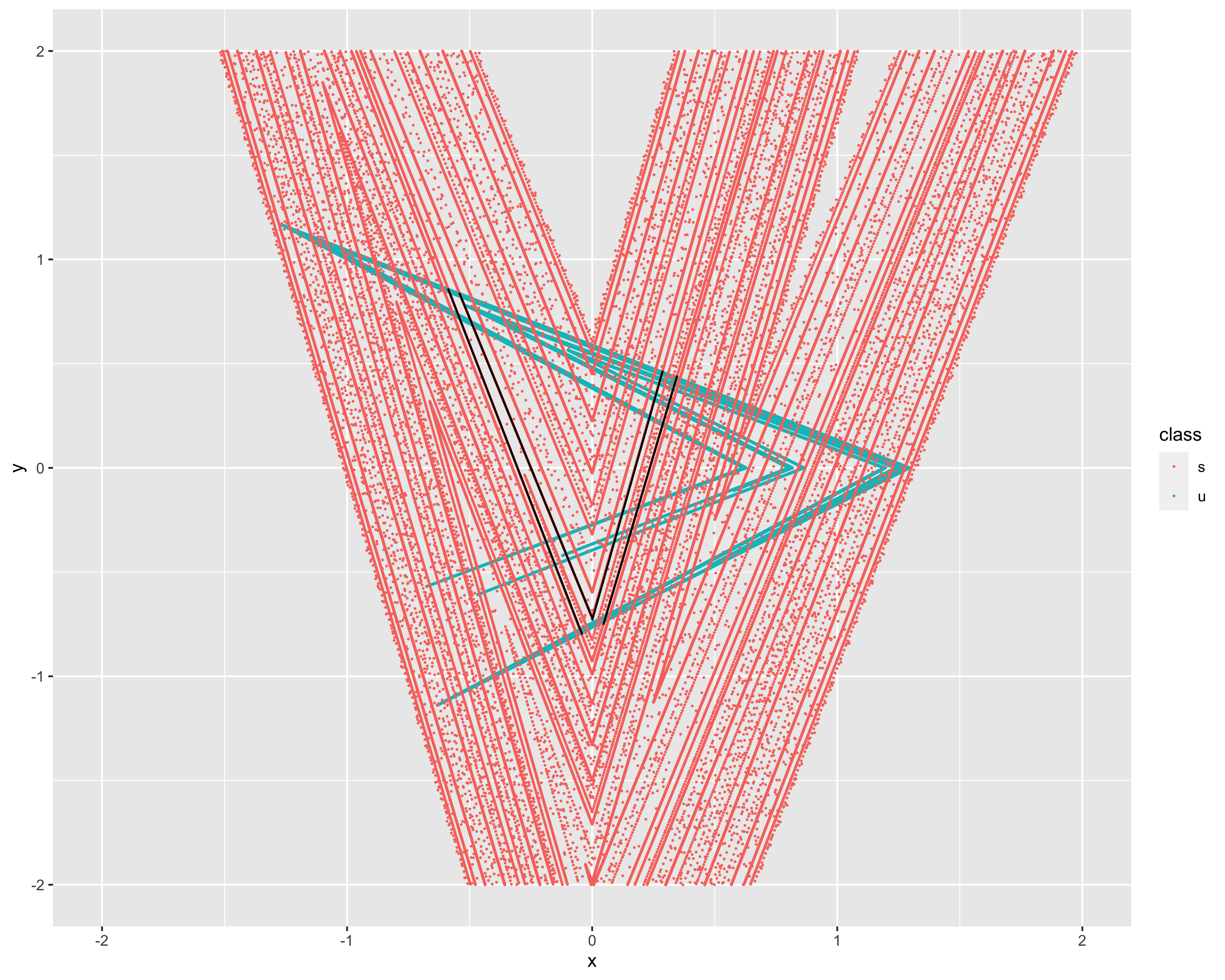}
		\caption{}
		\label{fig:NL}
	\end{figure}	
	Figure \ref{fig:NL} shows a computer image of (parts of)
	the stable (in red) and the unstable (in blue) manifolds of the fixed point $X$ of the Lozi map $L_{a,b}$ for $a = 1.75$ and $b = 0.45$. The stable boundary of a region $s$ is given in black. It seems that if there exists $\alpha \in \A$ such that $\pi(\alpha)$ is a branch point of order four, then that $\alpha$ would belong to a region such as $s$.
	
	The same questions can be asked for the H\'enon maps with parameters in the set $\mathcal{WY}$.

	\section{More on topology of Lozi and H\'enon strange attractors}\label{top}
	
	In this section we show one additional feature of strange attractors, which is topological indecomposability. Recall that a continuum $\mathcal{K}$ is indecomposable if given two distinct continua $\mathcal{K}',\mathcal{K}''\subset \mathcal{K}$ such that $\mathcal{K}'\cup\mathcal{K}''=\mathcal{K}$ we must have $\mathcal{K}'=\mathcal{K}$ or $\mathcal{K}''=\mathcal{K}$. There are several known examples of topologically indecomposable attractors in the literature, such as Smale's attracting solenoids, Plykin attractors, and DA-attractors. Any indecomposable continuum $\mathcal{K}$ admits the following decomposition. For any $x\in \mathcal{K}$, the \textit{composant} of $x$, denoted  $C(x)$, is the union of all proper subcontinua in $\mathcal{K}$ which contain $x$. We have
	\begin{itemize}
		\item if $x,y$ are points in $\mathcal{K}$, then either  $C(x)=C(y)$ or  $C(x)\cap C(y) = \emptyset$,
		\item there exists uncountably many composants,
		\item each composant is a dense first-category connected set in $\mathcal{K}$.
	\end{itemize}
	We start with the following lemma.
	\begin{lem}
		Suppose $H:\mathbb{R}^2\to\mathbb{R}^2$	is a homeomorphism, with a fixed point $p$. If for the unstable set $W^u_p$ of $p$ we have:
		\begin{enumerate}
			\item $W^u_p$ is a bounded subset of $\mathbb{R}^2$ which is a continuous one-to-one image of $\mathbb{R}$, 
			\item 
			$\Lambda=\Cl(W^u_p)$ does not separate the plane, and
			\item there exists a point $c\in W^u_p$, such that $\{H^n(c)\}_{n=0}^{\infty}$ is dense in $W^u_p$,
		\end{enumerate}
		then $\Lambda$ is indecomposable. 
	\end{lem}
	\begin{proof}
		Let $W^{u+}_p$ be the component of $W^{u}_p\setminus\{p\}$ that contains $\{H^{2n}(c)\}_{n=0}^{\infty}$. Note that $\Lambda$ is not an arc, since then $H|_{\Lambda}$ would not have been transitive. Let $\varphi:\mathbb{R}\to W^u_p$ be a parametrization of $W^u_p$ with $\varphi(0)=p$. We shall show that each proper subcontinuum of $\Lambda$ is nowhere dense in $\Lambda$, which is equivalent to indecomposability of $\Lambda$ (see \cite{HY}). Recall that any 1-dimensional nonseparating plane continuum is tree-like, and therefore hereditarily unicoherent; i.e.\ for any two subcontinua $A$ and $B$ we have that $A\cap B$ is a continuum (see e.g. \cite{FM}). 
		
		First observe that for every $n\in\mathbb{N}$ the arc $A_n=\varphi([0,n])$ contains only finitely many points from the orbit $\{H^{2n}(c)\}_{n=0}^{\infty}$. This is because $|\varphi^{-1}(H^{2i}(c))|<|\varphi^{-1}(H^{2j}(c))|$ for each $i<j$, since $H^2|_{W^{u+}_p}$ is an orientation-preserving homeomorphism. Consequently, each $A_n$ is nowhere dense in $\Lambda$; i.e.\ for every $x\in A_n$, and each open set $U_x$ around $x$, the neighborhood $U_x$ contains infinitely many points from  $\{H^n(c)\}_{n=0}^{\infty}$, so $U_x$ cannot be contained in $A_n$. 
		
		Suppose that $K$ is a subcontinuum of $\Lambda$ and that there exists an open set $U\subset \mathbb{R}^2$ such that $U\cap\Lambda\subset K$. Since for every $n\in\mathbb{N}$ the arc $A_n=\varphi([0,n])$ is a nowhere dense closed subset of $\mathbb{R}^2$, we must have that $U\cap (W^{u+}_p\setminus A_n)\neq\emptyset$. Consequently for every $n\in\mathbb{N}$ there exists an $H^{k_n}(c)\in K\cap (W^{u+}_p\setminus A_n)$. But since $\Lambda$ is unicoherent, for any $i,j\in\mathbb{N}$ the subarc in $W^{u+}_p$ with endpoints $H^{k_i}(c)$ and $H^{k_j}(c)$ is contained in $K$. It follows that $W^{u+}_p\subset K$ and so $\Lambda=\Cl(W^{u+}_p)\subset K$. This shows that the only subcontinuum of $\Lambda$ with nonempty interior is $\Lambda$ itself and completes the proof.
	\end{proof}
	In \cite{B}	Barge proved indecomposability of attractors of some $C^1$ diffeomorphisms under different assumptions and using different approach. 
	
	In \cite{M} Misiurewicz showed that $L_{a,b}|_{\Lambda}$ is mixing, for $(a,b) \in \M$, and this automatically implies transitivity. Below we extract from \cite{M} the location of some of the points with dense orbits.
	\begin{lem}
		For any segment $I\subset W^u_X$ there exists a point $c\in I$ such that $\{L^n(c):n\in\mathbb{N}\}$ is dense in $\Lambda$.
	\end{lem}
	\begin{proof}
		Let $V$ be an open subset of $\Lambda$ and fix a segment $I\subset W^u_X$. By \cite[proof of Theorem 5]{M} 
		\begin{equation}\label{eqn:N}
			\textrm{there exists an } N \in \N \textrm{ such that } L^n(I)\cap V\neq\emptyset \textrm{ for all } n \geq N. 
		\end{equation}
		Indeed, by \cite[proof of Theorem 5]{M} there exists an open set $V_1\subset V$ and integers $k_1,n_1$ such that $L^{n}(I)\cap V_1\neq\emptyset$ for all $n\geq n_1+k_1$, so it is enough to set $N=n_1+k_1$. 
		
		Now we adapt proof of \cite[Proposition 1.1]{Silverman}. Let $\mathcal{V}=\{V_n\}_{n=1}^\infty$ be a countable base of open sets for $\Lambda$. By contradiction suppose that a segment $J\subset W^u_X$ contains no point whose forward orbit is dense in $\Lambda$. Given $x\in J$ there exists a $V_{n_x}\in\mathcal{V}$ such that $L^k(x)\notin V_{n_x}$ for all $k\geq 0$. The set $\bigcup_{k=0}^\infty L^{-k}(V_{n_x})$ is open, and by (\ref{eqn:N}) for every segment $I\subset W^u_X$ there exists a positive integer $k$ such that $I\cap L^{-k}(V_{n_x})\neq\emptyset$, so $J\cap \left(\bigcup_{k=0}^\infty L^{-k}(V_{n_x})\right)$ is also dense in $J$. Consequently the set $A_{n_x}=J\setminus \left(\bigcup_{k=0}^\infty L^{-k}(V_{n_x})\right)$ is closed and nowhere dense in $J$, and contains $x$. Therefore $J=\bigcup_{x\in I}A_{n_x}$ and so $J$ is a countable union of nowhere dense subsets of $J$, a contradiction. 
	\end{proof}
	Since it is also well known that there is a dense orbit in H\'enon strange attractors contained in the unstable manifold of the fixed point, whose closure is the attractor we get the following (cf. Theorem in \cite{B}).
	\begin{theorem}
		For $(a,b)\in\mathcal{M}$ the strange attractor of the the Lozi map $L_{a,b}$ is indecomposable. The same holds for H\'enon map $H_{a,b}$ if $(a, b) \in \WY$.
	\end{theorem}

	\section{Acknowledgments}
	We are grateful to Sylvain Crovisier for his suggestion to study \cite{CP} in the context of Williams' result, encouragement and comments that helped to improve this paper. We are also indebted to Phil Boyland and Jernej \v Cin\v c for their helpful feedback on the first draft of this paper, as well as an anonymous referee for suggesting to highlight Lozi-like maps in the statements of results and proofs.

	\medskip
	\noindent
	J. P. Boro\'nski\\	
	Department of Differential Equations\\
	Faculty of Mathematics and Computer Science\\
	Jagiellonian University in Kraków\\
	ul. Łojasiewicza 6, 30-348 Kraków, Poland\\
	-- and --\\
	National Supercomputing Centre IT4Innovations\\ 
	IRAFM, University of Ostrava\\
	30. dubna 22, 70103 Ostrava, Czech Republic\\
	\texttt{jan.boronski@uj.edu.pl}\\
	
	\medskip
	\noindent
	Sonja \v Stimac\\
	Department of Mathematics\\
	Faculty of Science, University of Zagreb\\
	Bijeni\v cka 30, 10\,000 Zagreb, Croatia\\
	\texttt{sonja@math.hr}\\
	\texttt{http://www.math.hr/}$\sim$\texttt{sonja}
	

\begin{thebibliography}{99}
		
		\bibitem{BaV} D.\ Baptista, S.\ Vinagre, \emph{The basin of attraction of Lozi mappings}, 
		International Journal of Bifurcation and Chaos {\bf 19} (2009), 1043-1049.
		
		\bibitem{B} M.\ Barge, \emph{Homoclinic intersections and indecomposability}, Proc.\ Amer.\ Math.\ Soc.\  
		{\bf 101} (1987), 541–544.
		
		
		\bibitem{B2} M.\ Barge, S. Holte, \emph{Nearly one-dimensional H\'enon attractors and inverse limits}, 
		Nonlinearity {\bf 8} (1995), 29-42.
		
		
		
		\bibitem{BK} M.\ Benedicks, L.A.E.\ Carleson,  
		\emph{The dynamics of the H\'enon map}, Annals of Mathematics {\bf 133} (1991), 73--169.
		
		\bibitem{BV} M.\ Benedicks, M.\ Viana, \emph{Solution of the basin problem for Hénon-like attractors}, Inventiones Mathematicae, {\bf 143} (2001), 375-434.
		
		\bibitem{Berger} P. Berger, {\it Properties of the maximal entropy measure and geometry of H\'enon attractors.} J. Eur. Math. Soc. {\bf 21} (2019) 2233--2299.
		
		
		
		
		\bibitem{BiTy} J. Bishop, J.T. Tyson, {\it Conformal dimension of the antenna set.} Proc. Amer. Math.Soc. {\bf 129} (2001), 3631--3636
		
		\bibitem{BM} M.\ Bonk, D.\ Meyer, \emph{Uniformly branching trees}, preprint 2020, arXiv:2004.07912v1
		
		\bibitem{BT} M.\ Bonk, H.\ Tran, \emph{The continuum self-similar tree}, preprint 2020,	arXiv:1803.09694.
		
		\bibitem{BS2} J.\ Boro\'nski, S.\ \v Stimac, \emph{The pruning front conjecture, and a classification of Hénon maps in the presence of strange attractors}, preprint 2023,	arXiv:2302.12568.
		
		\bibitem{BCH}  P. Boyland, A. de Carvalho, T. Hall {\em Inverse limits as attractors in parameterized families} Bull. Lond. Math. Soc. {\bf 45} (2013), 1075--1085.
		
		\bibitem{BCH2} P. Boyland, A. de Carvalho, T. Hall {\em New Rotation Sets in a Family of Torus Homeomorphisms} Inv. Math. {\bf 204} (2016), 895--937.
		
		\bibitem{BCH3} P. Boyland, A. de Carvalho, T. Hall {\em Natural extensions of unimodal maps: virtual sphere homeomorphisms and prime ends of basin boundaries.} Geom. Topol. {\bf 25} (2021), 111--228.
		
		\bibitem{BCH4}  P. Boyland, A. de Carvalho, T. Hall {\em Statistical stability for Barge-Martin attractors derived from tent maps.} Discrete Contin. Dyn. Syst. {\bf 40} (2020), no. 5, 2903--2915.
		
		\bibitem{BCH5}  P. Boyland, A. de Carvalho, T. Hall {\em Typical path components in tent map inverse limits.} Fund. Math. {\bf 250} (2020), no. 3, 301--318.
		
		\bibitem{BS}
		M.\ Brin, G.\ Stuck, \emph{Introduction to Dynamical Systems}, Cambridge University Press, Cambridge, 2002.
		
		\bibitem{C} Y.\ Cao, \emph{The transversal homoclinic points are dense in the codimension-1 Henon-like strange
			attractors}, Proc.\ Amer.\ Math.\ Soc.\   {\bf 127} (1999),  1877-1883. 
		
		\bibitem{CP} S.\ Crovisier, E.\ Pujals, \emph{Strongly dissipative surface diffeomorphisms},
		Commentarii Mathematici Helvetici {\bf 93} (2018), 377–400. 
		
		\bibitem{CPT} S.\ Crovisier, E.\ Pujals, C.\ Tresser, \emph{Mildly dissipative diffeomorphisms of the disk with zero entropy}, preprint 2020, arXiv:2005.14278.
		
		\bibitem{FM} J.B.\ Fugate, L.\ Mohler, \emph{The fixed point property for tree-like continua with finitely many arc components}, Pacific J.\ Math.\  {\bf 57} (1975), 393–402.
		
		\bibitem{H} M.\ H\'enon, \emph{A two-dimensional mapping with a strange attractor},
		Communications in Mathematical Physics, {\bf 50}, no.\ 1 (1976), 69-77.
		
		\bibitem{HY} J.G.\ Hocking, G.S.\ Young, \emph{Topology}, Addison-Wesley Publishing Co., Inc., Reading, Mass.-London 1961 ix+374 pp.
		
		
		
		\bibitem{I} Y.\ Ishii, \emph{Towards a kneading theory for Lozi mappings. I. A solution of the pruning front conjecture and the first tangency problem}, Nonlinearity {\bf 10} (1997), no. 3, 731-747. 
		
		\bibitem{I2} Y.\ Ishii, \emph{Towards a kneading theory for Lozi mappings. II. Monotonicity of the topological entropy and Hausdorff dimension of attractors}, Comm.\ Math.\ Phys.\ {\bf 190} (1997), no.\ 2, 375-394.
		
		\bibitem{IS} Y.\ Ishii, D.\ Sands, \emph{Monotonicity of the Lozi family near the tent-maps}, Comm.\ Math.\ Phys.\ {\bf 198} (1998), no.\ 2, 397-406.
		
		\bibitem{Kig} J. Kigami {\em Analysis on Fractals.} Vol. 143 of Cambridge Tracts in Mathematics, Cambridge University Press, Cambridge, 2001
		
		\bibitem{L}
		R.\ Lozi, \emph{Un attracteur etrange(?) du type attracteur de H\'enon}, J.\ Physique (Paris) {\bf 39} (Coll. C5) (1978), 9--10.
		
		
		\bibitem{M}
		M.\ Misiurewicz, \emph{Strange attractor for the Lozi mappings}, Ann.\ New York Acad.\ Sci.\ {\bf 357} (1980) (Nonlinear Dynamics), 
		348--358.
		
		\bibitem{MS}
		M.\ Misiurewicz, S.\ \v Stimac, \emph{Symbolic dynamics for Lozi maps}, Nonlinearity {\bf 29} (2016), 3031--3046.
		
		\bibitem{MS2}
		M.\ Misiurewicz, S.\ \v Stimac, \emph{Lozi-like maps}, Discrete and Continuous Dynamical Systems - Series A {\bf 38} (2018), 2965-2985.  
		
		\bibitem{MV}
		L.\ Mora, M.\ Viana, \emph{Abundance of strange attractors}, Acta Math.\ {\bf 171} (1993), 
		1--71.  
		
		
		\bibitem{O} D.-S.\ Ou, \emph{Nonexistence of wandering domains for strongly dissipative infinitely renormalizable Hénon maps at the boundary of chaos}, Invent.\ Math.\ {\bf 219} (2020), no.\ 1, 219-280.
		
		\bibitem{Rohlin} V.  A.  Rohlin.  {\it Lectures  on  the  entropy  theory  of  transformations  with  invariant measure.} Uspehi Mat. Nauk, 22 (1967), 3--56.
		
		\bibitem{Silverman} S.\ Silverman, \emph{On maps with dense orbits and the definition of chaos}, Rocky Mountain J.\ Math.\ {\bf  22} (1992), 353-375.
		
		\bibitem{S}  S.\ van Strien, \emph{One-dimensional dynamics in the new millennium}, Discrete and Continuous Dynamical Systems - Series A {\bf 27} (2010), 557-588.
		
		\bibitem{WY} Q.\ Wang, L.-S.\ Young, \emph{Strange attractors with one direction of instability}, Comm.\ Math.\ Phys.\ {\bf 218} (2001), no.\ 1, 1-97.
		
		\bibitem{Wazewski} T. Wazewski, {\em Sur les courbes de Jordan ne renfermant aucune courbe simple ferme\'e de Jordan.} Ann. Soc. Polonaise Math. {\bf 2} (1923), 49--170.
		
		\bibitem{W}
		R.F.\ Williams, \emph{One-dimensional non-wandering sets}, Topology {\bf 6} (1967), 473-487.  
		
		\bibitem{Y} I.\ B.\ Yildiz, \emph{Monotonicity of the Lozi family and the zero entropy locus}, Nonlinearity {\bf 24} (2011), 1613-1628.
	\end{thebibliography}
\end{document}